\documentclass[11pt]{amsart}

\usepackage[margin=1in, marginparwidth=2cm]{geometry} 
\usepackage{changepage} 
\usepackage{enumitem}   
\usepackage{appendix}   

\usepackage{mathtools}  
\usepackage{amssymb}    
\usepackage{mathrsfs}   
\usepackage{diffcoeff}  

\usepackage{graphicx}   
\usepackage[all]{xy}    

\usepackage{tikz}
\usepackage{tikz-cd}    

\usetikzlibrary{
    patterns,
    arrows.meta,
    bending,
    calc,
    intersections
}

\usepackage{cite}       
\usepackage[colorlinks=true, linkcolor=blue, citecolor=green, urlcolor=black]{hyperref} 

\usepackage{marginnote}
\usepackage[colorinlistoftodos, textsize=tiny]{todonotes}

\newcommand{\bea}[1]{\marginpar{\tiny \color{blue} Bea: #1}} 
\newcommand{\jia}[1]{\marginpar{\tiny \color{red} Jiajun: #1}}

\theoremstyle{plain} 
\newtheorem{thm}{Theorem}[section]
\newtheorem{lemma}[thm]{Lemma}
\newtheorem{cor}[thm]{Corollary}
\newtheorem{prop}[thm]{Proposition}

\theoremstyle{definition} 
\newtheorem{defn}[thm]{Definition}
\newtheorem{ex}[thm]{Example}

\theoremstyle{remark} 
\newtheorem*{rmk}{Remark}

\newcommand{\mb}{\mathbb}

\newcommand{\mc}{\mathcal}

\newcommand{\Z}{\mathbb{Z}}
\newcommand{\R}{\mathbb{R}}
\newcommand{\C}{\mathbb{C}}

\newcommand{\N}{\mathbb{N}}

\newcommand{\inv}{^{-1}} 

\DeclarePairedDelimiter{\paren}{(}{)}           
\DeclarePairedDelimiterX{\innerproduct}[2]{\langle}{\rangle}{#1, #2}
\DeclarePairedDelimiter{\abs}{\lvert}{\rvert}   
\DeclarePairedDelimiter{\norm}{\lVert}{\rVert}  
\DeclarePairedDelimiterX{\set}[2]{\{}{\}}{#1 \mathrel{}\delimsize\vert\mathrel{} #2} 

\DeclareMathOperator{\CAT}{CAT}
\newcommand{\supp}[1]{\operatorname{supp}(#1)}

\newcommand{\til}{\widetilde}
\newcommand{\wS}{\widetilde{\Sigma}}
\newcommand{\wt}{\widetilde}
\newcommand{\wideS}{\widetilde{\Sigma}}
\newcommand{\wideF}{\widetilde{\mathcal{F}}}

\newcommand{\calF}{\mathcal{F}}
\newcommand{\calC}{\mathcal{C}}

\title{Finsler metrics on $1/n$-translation structures on surfaces}
\author{Beatrice Pozzetti and Jiajun Shi}
\date{\today}

\begin{document}

\begin{abstract}
We define compatible Finsler distances on $1/n$-translation surfaces, we study their geodesics, and construct a Liouville current for each such metric, that is a geodesic current that encodes the information of the length of the closed curves. The construction is based on multi-foliations, a generalization of measured foliations of independent interest.
\end{abstract}

\maketitle

\tableofcontents

\section{Introduction}

The goal of the paper is to study Finsler metrics on singular flat surfaces. More specifically for every $n\in \N$, we consider \emph{$1/n$-translation structures} on compact surfaces. These can be constructed  from a collection of Euclidean polygons by identifying pairs of edges via translation and rotations of order $n$ to obtain a topological surface $\Sigma$ with a singular flat geometric structure. Vertices of polygons typically give rise to singularities called \emph{cone points}. We require all cone angles to be at least $2\pi$, which forces the genus $g$ to be at least $2$ by Gauss--Bonnet Theorem. Examples of such surfaces are translation surfaces, which are encompassed in the case $n=1$ and whose theory is linked with the study of abelian differentials on surfaces \cite{wright2014translationsurfacesorbitclosures}, and the so-called half-translation surfaces, arising from quadratic differentials. Both the study of abelian and half-translation surfaces are central in Teichmüller theory and more generally in low dimensional geometry and topology.

In a $1/n$-translation surface, the tangent space at any regular point, namely any point that is not in the image of one of the vertices, is well-defined and can be identified with $\R^2$ in a way that is canonical up to a rotation of angle multiple of $2\pi/n$. In particular any choice of a norm $\norm\cdot$ on $\R^2$ that is invariant under rotations of order $n$ naturally induces a \emph{Finsler metric} $F$  on any $1/n$ translation surface. Special instances of these metrics, more specifically the metric arising from the hexagonal norm on $\R^2$ for  1/3-translation surfaces, and that induced from the $\ell^1$-norm in the case of 1/2- and 1/4-translation surfaces, have played an important role in the study of compactifications of character varieties, see  \cite{burger2021weyl, burger2023realspectrumcompactificationcharacter, ouyang2023length, Reid}. Even in the case of  translation surfaces arising from abelian differentials, the study of Finsler distances is intriguing and not well explored.  In this paper we study these metrics, with particular focus on the length of closed curves.


The Finsler metric $F$ induces  a distance function $d^F$. When the associated norm is not strictly convex, the induced metric on $1/n$-translation surfaces  is typically not uniquely geodesic and even worse, local geodesics in the same homotopy class might have different length. Our first result explicitly constructs a bi-combing, namely a choice of  geodesics between any pair of points, which  realizes the minimal length in the homotopy class. Our bi-combing does not depend on the Finsler metric but only on the $1/n$-translation structure. More specifically, we know, using CAT(0) geometry, that every homotopy class relative to endpoints admits a unique representative which is a concatenation of straight segments that meet at singular points and make an angle at least $\pi$ on both sides. Our bi-combing is obtained by choosing such path in each equivalence class. Calling such paths the \emph{piecewise-straight} paths, we prove:

\begin{thm}\label{thmA}
    For any  norm $\norm\cdot$ on $\R^2$ invariant by rotations of order $n$, and every $1/n$-translation structure $\Sigma$, the piecewise-straight paths are geodesics for $d^F$ and
    realize the smallest length in (relative) homotopy classes.
\end{thm}

Our proof of Theorem \ref{thmA} is geometric: we  construct a norm-non-increasing projection of the universal cover $\widetilde \Sigma$ to a piecewise-straight path. The construction is based on what we call \emph{brush lines}, a partial foliation of  $\widetilde \Sigma$ that depends both the $1/n$-translation structure and the  norm.

A natural question about geometric structures is to understand the smallest length of closed curves in a homotopy class. In the case of hyperbolic structures on surfaces, Bonahon introduced \emph{geodesic currents} as an efficient tool to address this problem \cite{bonahon1988geometry}. These are $\pi_1(\Sigma)$-invariant, flip-invariant measures on pairs of distinct points in $\partial\wideS$. Both curves and negatively curved metrics on surfaces naturally give rise to geodesic currents. Thanks to Bonahon's intersection, a generalization of the geometric intersection of curves, a geodesic current gives a way of measuring the length of closed curves. It is a classical result of Otal \cite{otal1990spectre} that every negatively curved metric on a surface admits a \emph{Liouville current}, namely a current whose associated length function agrees with the length function induced by the metric. Similar results were obtained by Duchin--Leininger--Rafi for the CAT(0) metric associated to a quadratic differential \cite{duchin2010length},  and by Bankovich--Leininger for general singular flat structures without restriction on the holonomy \cite{bankovic2018marked}.

Inspired by the construction in \cite{duchin2010length}, we show that all the Finsler metrics we introduce in the paper also admit Liouville currents:

\begin{thm}\label{thm:liouville}
    There exists a Liouville current for each $1/n$-translation structure with any compatible Finsler metric.
\end{thm}
The currents we construct depend both on the $1/n$-translation structure and on the chosen  norm in a fairly explicit way, giving tools to understand their support, as well as typical geodesics.

More specifically we prove that, for any  norm $\|\cdot\|$ on $\R^2$, inducing a Finsler metric $d^F$ on a  $1/n$-translation surface $\Sigma$, the Liouville current of $d^F$ belongs to the convex hull of the set of \emph{measured multi-foliations currents} for $\Sigma$, special geodesic currents that we introduce and play the role of measured laminations for translation surfaces. By means of convex geometry we describe how to determine, knowing $\|\cdot\|$, which combinations of the measured multi-foliations correspond to the Liouville current for the distance $d^F$.

We conclude this introduction by discussing more in detail the measured multi-foliations currents and their support. A measured multi-foliation is determined by an $1/n$-translation structure $\Sigma$ and a choice of angle $\theta$, and, loosely speaking, corresponds to the union of the straight lines at angles $\theta+2\pi k/n$. In Section \ref{s.mutheta2} we use these singular foliations to construct a geodesic current $\mu_\theta$ on $\Sigma$, and we show that if the  $1/n$-translation surface is induced by the $n$-th power of an abelian differential, the $\theta$-multi-foliation is the sum of the measured laminations corresponding to the $n$ relevant angles (Corollary \ref{degenerate multi-foliation}). In general, when $n\ge 3$, measured multi-foliations are associated with particularly simple Finsler metrics (Corollary \ref{theta length}) and as a first step in the proof of Theorem \ref{thm:liouville} we show that the multi-foliations currents are the Liouville currents for the associated metric (Theorem \ref{foliation intersection}). In the case of cubic differentials the multi-foliation current $\mu_\theta$ was independently constructed by Charlie Reid in \cite[Section 7]{Reid} with different techniques.

In fact our multi-foliation currents are naturally  oriented geodesic currents, namely measures on $\partial\widetilde\Sigma^2$ that avoid the diagonal (see Section \ref{s.mutheta2} for details). Generalizing work of Charlie Reid in higher rank Teichmüller theory, we get some concrete bounds on the possible intersection patterns of geodesics in the support of $\mu_\theta$. To make this precise, we say $k$ pairs $(x_1,y_1),\ldots,(x_k,y_k)\in \partial\widetilde\Sigma^2$ form a \emph{positive k-crossing} if the $2k$-tuple $(x_1,\ldots, x_k,y_1,\ldots,y_k)\in \partial\widetilde\Sigma^{2k}$ is counterclockwise oriented.

\begin{thm}\label{thmINTRO:crossing}
Any geodesic in $\supp\mu$ is contained in a positive $\lfloor n/2\rfloor$-crossings contained in $\supp{\mu_\theta}$. No positive $(\lfloor n/2\rfloor+1)$-crossing is contained in the support of the $\theta$-multi-foliation current $\mu_\theta$ associated to a $1/n$-translation surface. 
\end{thm}

In fact, any geodesic in $\supp{\mu_\theta}$ is contained in many positive $\lfloor n/2\rfloor$-crossings contained in $\supp\mu$ (see the proof of Theorem \ref{thm:crossing_free} for details). Furthermore, the set $\supp{\mu_\theta}$ is maximal among subsets avoiding positive $(\lfloor n/2\rfloor +1)$-crossing in the following sense: for any $(x,y)\notin\supp{\mu_\theta}$ whose geodesic is not always tangent to the $\theta$-multi-foliation, the set $\supp{\mu_\theta}\cup \{(x,y)\}$ contains a positive $(\lfloor n/2\rfloor +1)$-crossing (see Theorem \ref{thm:crossing_free}). 
Theorem \ref{thmINTRO:crossing} follows from a precise understanding of the support of $\mu_\theta$ (Proposition \ref{p.support}). That  no positive $(\lfloor n/2\rfloor+1)$-crossing is contained in the support of $\mu_\theta$ generalizes the property of measured laminations of having vanishing self-intersection, and is in strong contrast with Liouville currents of negatively curved metrics, whose support contains positive $k$-crossings for all $k$. 

It follows from results of Loftin--Tamburelli--Wolf \cite{LTW} and Reid \cite{reid0} that  the limit of the currents associated to representations in cubic rays in the Hitchin component in ${\rm SL}_3(\mathbb R)$ is the  horizontal multi-foliations currents for the 1/3-translation structure associated to the cubic differential. Using this analytic result as a black-box, we obtain, as a corollary of Theorem \ref{thmINTRO:crossing}, an entirely different geometric proof of a special case of \cite[Theorem 1]{Reid}.

We expect that in general the precise  information about the support of the Liouville currents will provide important tools to study rigidity and flexibility questions.

\addtocontents{toc}{\protect\setcounter{tocdepth}{1}}
\subsection*{Comparison with other results}Our construction is inspired from and generalizes the construction of Duchin--Leininger--Rafi for the CAT(0) metric associated to a quadratic differential \cite{duchin2010length}. In the special case of the asymmetric triangular norm associated to cubic differentials on surfaces, Reid constructs Liouville currents in \cite{Reid} via cross-ratio over points in the horoboundary of the Finsler metric. 

\subsection*{Structure of the paper}In Section \ref{sec:def_section} we introduce the main objects of study, namely $1/n$-translation surfaces and compatible Finsler metrics. We also give a basic example of Finsler metrics induced by $\theta$-web (see Definition \ref{web_definition}). In Section \ref{sec:CAT(0)_geodesic}, we show that CAT(0) geodesics realize the minimal length in the homotopy class by constructing a norm non-increasing projection. In Section \ref{sec:multifoliation} we define measured multi-foliations, and construct their assiciated oriented geodesic current. In Section \ref{sec:propertiesmulti} we discuss the intersection of multi-foliation currents and closed curves, and prove Theorem \ref{thmINTRO:crossing} on positive $k$-crossings in the support of $\mu_\theta$. In Section \ref{sec:Liouville_current}, we construct the Liouville current for a generic Finsler metric by studying convex combinations of measured multi-foliations.
\addtocontents{toc}{\protect\setcounter{tocdepth}{2}}

\section{Singular Finsler metrics on $1/n$-translation surfaces}\label{sec:def_section}
We  discuss in this section  $1/n$-translation surfaces from the viewpoint of $(G,X)$-structures (Section \ref{sec:nTS}) and consider the associated CAT(0) metric, which will be useful also in the study of general Finsler metrics. In Section \ref{sec:Finsler} we endow $1/n$-translation surfaces with Finsler metrics.

\subsection{The $1/n$-translation surfaces and their CAT(0) metric}\label{sec:nTS}
A \emph{$(G,X)$-structure} on a manifold $M$, where $X$ is a manifold acting as model space and $G$ is a group acting on $X$ by homeomorphisms, is the datum of an atlas on $M$ whose charts are homeomorphisms onto open subsets of $X$, and whose change of charts are restrictions of elements of $G$. In this paper we will consider  $(\R^2 \rtimes \Z/n\Z, \R^2)$-structures for integers $n\in\N$: here $\R^2 \rtimes \Z/n\Z$ is the subgroup of the group of Euclidean isometries of $\R^2$ on which $\R^2$ acts as translations and $\Z/n\Z$ is the group generated by a $\frac{2\pi}{n}$-rotation. 

\begin{defn}
A topological surface $\Sigma$ has a $1/n$-\emph{translation structure} if there exists a discrete set of points $P\subset\Sigma$ and an $(\R^2 \rtimes \Z/n\Z, \R^2)$-structure on $\Sigma \setminus P$ such that every $p\in P$ has a neighborhood homeomorphic to a neighborhood of a cone point of cone angle $\frac{2k\pi}{n}$ for some $k>n$, which is an isometry away from $p$.   We refer to points in $P$ as \emph{cone points} or \emph{singular points}.
\end{defn}
If a surface admits a $1/n$-translation structure then it is oriented. We will mostly focus on the case in which either $\Sigma$ is compact, or the $1/n$-translation structure induced on the universal cover of a compact $1/n$-translation surface. Examples of $1/n$-translation surfaces can be obtained by glueing families of polygons in the plane with translations and rotations by angles multiple of $2\pi/n$. In fact every $1/n$-translation surface arises this way. See Figure \ref{fig:examples} for some explicit examples. 

\begin{figure}[h]
    \centering
    \includegraphics[width=0.75\linewidth]{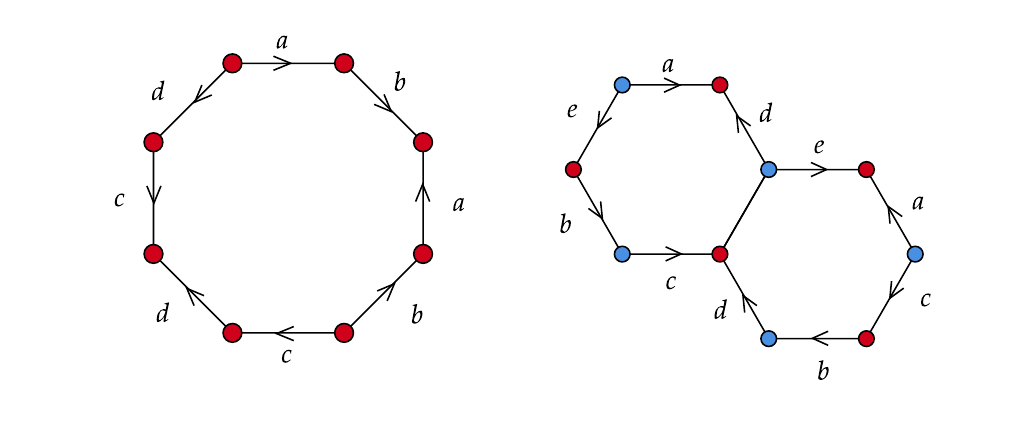}
    \caption{A $\frac14$-translation surface on the left and a $\frac 13$ translation surface on the right. Sides with the same label are identified isometrically. The surface on the left has one cone point of angle $6\pi$, the surface on the right has two cone points of angles $4\pi$, both surfaces have genus two. }
    \label{fig:examples}
\end{figure}
By a combinatorial version of Gauss-Bonnet Theorem (see \cite[Theorem 17.1]{schwartzMostlySurfaces2011} for a proof), if $\Sigma$ is closed, and $\theta(p)$ is the cone angle at $p$, we have the following formula:
\begin{equation}\label{eq:GB}
    \sum_{p\in P}(\theta(p)-2\pi)=4\pi(g-1).
\end{equation}
    In particular, $g\ge 2$ and there are at most $2n(g-1)$ cone points as $\theta(p)-2\pi\ge \frac{2\pi}{n}$.

 The $(\R^2 \rtimes \Z/n\Z, \R^2)$-structure equips $\Sigma\setminus P$ with a Euclidean metric. We refer to this Euclidean metric with conical singularities as a \emph{flat cone metric}. The requirement on cone angles implies that the flat cone metric is locally CAT(0), and its universal cover is CAT(0).

\begin{defn}\label{defn:CAT(0)_geodesic}
   The \emph{distance} between two points $p,q$ in a $1/n$-translation surface is the infimum of the length of piecewise $C^1$ curves joining $p$ and $q$. A \emph{geodesic} is a curve $\alpha:I\to\Sigma$ which locally realizes the distance, where $I$ is a connected subset of $\R$. We say that a geodesic is \emph{regular} if it doesn't contain cone points in its interior, \emph{singular} otherwise. A \emph{saddle connection} is a regular geodesic whose endpoints are cone points. 
\end{defn}
Note that, in contrast with the cases of abelian differential and quadratic differential where a saddle connection is embedded, in our case, saddle connections are generally just immersed.

Since the metric is locally Euclidean it is easy to characterize geodesics. The completeness assumption always holds if $\Sigma$ is compact or if it the universal cover of a compact translation surface.

\begin{prop}[See {\cite[\S 2.4]{duchin2010length}} for $n=2$]\label{geodesic behavior}
  Let $\Sigma$ be a complete $1/n$-translation surface with its flat cone metric, then:
    \begin{enumerate}
        \item Each homotopy class of curves (relative to endpoints for non-closed curves) admits a geodesic representative.
        \item Geodesics are concatenations of saddle connections and possibly geodesic rays such that two adjacent segments make angles at least $\pi$ on both sides at the common point.
    \end{enumerate}
\end{prop}

\begin{figure}[h]
    \centering
    \includegraphics[width=0.5\linewidth]{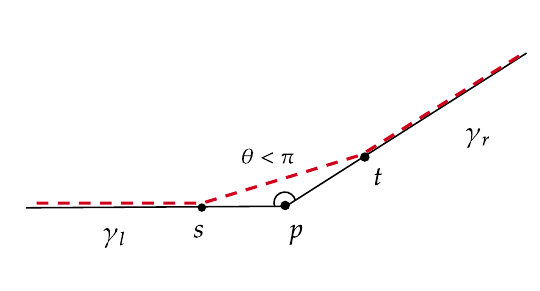}
    \caption{If one angle formed by $\gamma_l,\gamma_r$ is less than $\pi$, then the length of red curve is smaller than the sum of the length of $\gamma_l,\gamma_r$}
    \label{angle less pi}
\end{figure}

\begin{proof} 
1. Let $[\gamma]$ be a homotopy class of closed curves (or a homotopy class of geodesic segments relative to endpoints). Any sequence of  of piecewise $C^1$ curves  $\gamma_n\in[\gamma]$ with $\lim_{n}\ell(\gamma_n)=\inf_{\eta\in[\gamma]}\ell(\eta)$ has uniformly bounded length. 
After reparametrization, we may assume that the domain of definition of $\gamma_n$ is the unit interval and the curve has constant speed $\ell(\gamma_n)$. Moreover, we may assume all $\gamma_n$ lie in a compact set, either the compact surface itself or a fundamental domain in the universal cover. By Arzela-Ascoli,  there exists a subsequence converging uniformly to a curve, which is the seeked geodesic.

2. Consider an arbitrary geodesic $\gamma=\gamma(t)$. If $p=\gamma(t_0)$ is a regular point, then it has a  Euclidean neighborhood, and for sufficiently small $\epsilon$, the intersection of the interval $\paren{\gamma(t_0-\epsilon),\gamma(t_0+\epsilon)}$ with the neighborhood must be an Euclidean geodesic, namely a straight line. In particular, a saddle connection is locally an immersed straight line. If $p$ is a singular point on $\gamma$, the geodesic is locally concatenation of two regular geodesics $\gamma_l,\gamma_r$ forming two angles at $p$, each of which must be at least $\pi$: if one angle were acute, we would  find points $s,t$ close to $p$ on $\gamma_l,\gamma_r$ contained in a common Euclidean chart where they could be joined by a shorter path than the concatenation of $\gamma_l,\gamma_r$ by triangle inequality, contradicting the definition of a geodesic.
\end{proof}

We collect here some useful additional properties of locally CAT(0) metric on a closed surface. We say that a \emph{geodesic triangle} is a topological disk whose boundary is formed by three geodesics meeting at its vertices. 

\begin{prop}\label{CAT0 property}
    Suppose $(\Sigma,g)$ is a closed surface with a locally $\CAT(0)$ metric $g$, then: 
    \begin{enumerate}
        \item The universal cover $\wideS$ is homeomorphic to $\R^2$ and $g$ is lifted to a $\CAT(0)$ metric on $\wideS$ \cite[Corollary II 1.5, Theorem II 4.1]{bridson2013metric};
        \item  The sum of the interior angles of a geodesic triangle 
        is at most $\pi$ \cite[Proposition II 1.7]{bridson2013metric}. In particular, there is no bigon.
    \end{enumerate}
\end{prop}

 We conclude the subsection discussing properties of regular closed geodesics, namely closed geodesics that do not contain singular points in its interior, in analogy with translation surfaces we refer to these as \emph{cylinder curves}. Note however that, while if $n\leq 2$, cylinder curves are simple and contained in flat cylinders embedded in the surface, if $n\geq 3$ a cylinder curve might have non-trivial self-intersections.

\begin{prop}\label{p.closed_conepoint}
    Every cylinder curve $\gamma$ is contained in a one parameter family of parallel curves, forming an immersed cylinder. The two extremal curves in the family are singular.
\end{prop}

\begin{proof}
    Let $\gamma$ be a cylinder curve. Since the cone point set $P$ is discrete, the distance $d(\gamma, P)=\inf_{p\in P}d(\gamma,p)$ for a cylinder curve  is strictly positive. Consider an oriented lift $\widetilde\gamma$ of $\gamma$ to the universal covering $\wideS$. The $d$ neighbourhood of $\widetilde\gamma$ does not contain cone points, and is thus isometric to a flat strip on which the element in the deck transformation group corresponding to $\gamma$ acts by translations. The desired neighbourhood is the image in $\Sigma$ of the largest flat strip containing $\widetilde \gamma$. Its boundary is constituted of singular geodesics by maximality. 
\end{proof}

The following lemma about cylinders will be useful for the discussion in Section \ref{sec:multifoliation}.

\begin{prop}[{\cite[Corollary 2.3]{shi2024liouville}}]\label{p.bound_cylinder}
    If two geodesics in the universal cover stay bounded distance to each other, then they project  to cylinder curves and bound a common cylinder.
\end{prop}

\subsection{Finsler metrics}\label{sec:Finsler}
In this subsection we introduce Finsler metrics compatible with $1/n$-translation structure. We will show by examples that Finsler metrics on a $1/n$-translation surface are in general not uniquely geodesic.

\begin{defn}
    A \emph{norm} on a vector space $V$ is a continuous map $\norm \cdot:V\to\R$ such that 
    \begin{itemize}
        \item (positivity) $\norm v\ge 0$, for every $v\in V$, and equality holds if and only if $v=0$.
        \item (positive homogeneity) $\norm{\lambda v}=\lambda \norm v$, for every $v\in V,\, \lambda\ge 0$.
        \item (triangle inequality) $\norm{v+w}\le \norm v+ \norm w$, for every $ v,w\in V$.
    \end{itemize}
\end{defn}

It is easy to see that $\norm\cdot$ is a norm on $V$ if and only if the unit ball 
$$B_{\norm\cdot}^1:=\{v\in V:\norm v \le1\}$$
is a star shaped convex set  in $V$, namely  that for any $ v,w\in B_{\norm\cdot}^1$ and any  $t\in (0,1)$, $tv+(1-t)w\in B_{\norm\cdot}^1$ and for any $t\in[0,1)$,  $t v\in B_{\norm\cdot}^1$. If $Q\subset\mathbb R^2$ is a star-shaped convex set, we will denote by $\norm\cdot_Q$ the associated norm.

    A norm is  \emph{symmetric} if $\norm{-v}=\norm v$, or equivalently if its unit ball is invariant under the isometry $-{\rm Id}$, and is
    \emph{strictly convex} if $\norm {v+w}<\norm v+\norm w$ unless $v=\lambda w$ for some $\lambda>0$, or equivalently if its unit ball is strictly convex.
We will mostly consider symmetric  norms, but  non-strictly convex ones will play an important role.
%
A  norm $\norm\cdot$ is  induced by an inner product if and only if its unit sphere is an ellipse: this is equivalent to $\norm\cdot$ satisfying the parallelogram identity.

\medskip

A $1/n$-translation structure on a surface $\Sigma$ induces an identification of the tangent space $T_x\Sigma$ at any regular point $x$ with $\mathbb R^2$, which is determined up to rotations of angles $2\pi/n$.
\begin{defn}\label{d.cFinslerm}
    A \emph{Finsler metric} on a $1/n$ translation surface $\Sigma$ is a continuous map $$F:T(\Sigma\setminus P)\to \R$$ such that for every $ x\in \Sigma$, $F|_{T_x\Sigma}$ is a  norm. We  say that $F$ is \emph{induced from a $\mathbb Z/n\mathbb Z$-invariant norm $\|\cdot\|_Q$}  if, under the standard identification  $T_x\Sigma\cong\mathbb R^2$, $F|_{T_x\Sigma}=\|\cdot\|_Q$. In this case we say that the Finsler metric is \emph{compatible} with the $1/n$-translation structure. 
\end{defn}
In this paper we will only consider compatible Finsler metrics: these are very special elements within the class of Finsler metrics; however, as they can be equivalently characterized as being the pull-back of Finsler metrics on  the manifold $\R^2$ that are invariant under the group $\R^2\rtimes \Z/n\Z$, they are natural in the context of $(G,X)$-structures. 
%
%
%
%
%
%

We will make precise in Section \ref{sec:Liouville_current} that the  Finsler metrics induced by the following norms can be used as building blocks to construct any compatible Finsler metric. See Figure \ref{fig:web} for examples.

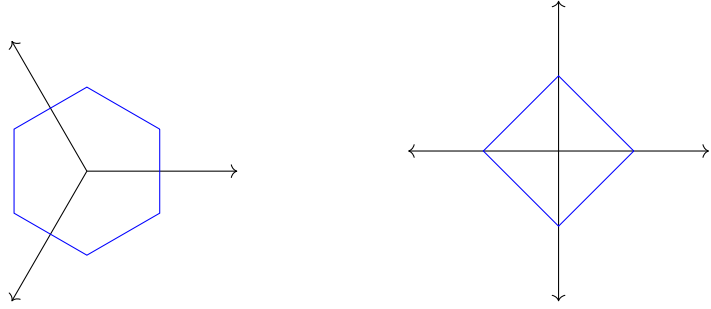
\begin{figure}[h]
\begin{center}
\begin{tikzpicture} 
    \draw[->] (0:0) to (0:2);
    \draw[->] (0:0) to (120:2);
    \draw[->] (0:0) to (240:2);  
    \draw[blue] (30:{sqrt(5)/2}) to (90:{sqrt(5)/2})to (150:{sqrt(5)/2}) to (210:{sqrt(5)/2}) to (270:{sqrt(5)/2}) to (330:{sqrt(5)/2}) to (30:{sqrt(5)/2});
\end{tikzpicture}
\hspace{2cm}
\begin{tikzpicture} 
    \draw[->] (0:0) to (0:2);
    \draw[->] (0:0) to (90:2);
    \draw[->] (0:0) to (180:2);        \draw[->] (0:0) to (270:2);
    \draw[blue] (0:1) to (90:1) to (180:1) to (270:1) to (0:1);
\end{tikzpicture}
\end{center}
\caption{The black vectors form the 3-web at angle 0, and the 4-web at angle 0, and the blue polygons are the unit spheres of the associated norms.}\label{fig:web}
\end{figure}

\begin{defn}\label{web_definition}\
    \begin{itemize}
        \item The $\theta$-\emph{web} is the $n$-tuple of vectors $v_{[\theta]}=\{\exp(i(\theta+2\pi k/n)), \;k=0,1,\ldots,n-1\}$ in  $\R^2\cong \C$.
        \item The \emph{$\theta$-norm} is 
        \[
    \|u\|_\theta=\sum_{w\in v_{[\theta+\frac\pi 2]}} \abs{\innerproduct{u}{w}}.
    \]
      %
    \end{itemize}
\end{defn}
Note that the $\theta$-norm is obtained by considering absolute values of inner products with the vectors in the \emph{orthogonal} $\theta$-web because this is the norm that will play a role in the construction of measured multi-foliations.

For every $1/n$-translation surface $\Sigma$ we will still denote by $v_{[\theta]}$ the $\theta$-web in $T(\Sigma\setminus P)$, namely the field of $n$-vectors  induced by the $\theta$-web in $\R^2$ under the identification $T(\Sigma\setminus P)\cong\R^2$. 
We will refer to the Finsler metric $F_\theta$ induced by the $\theta$-norm $\|\cdot\|_\theta$ on $\mathbb R^2$ as the \emph{$\theta$-metric}. It holds
\begin{equation}\label{e.Finsler_example}
    F_\theta(u)=\sum_{w\in v_{[\theta+\frac \pi 2],p}} \abs{\innerproduct{u}{w}}
\end{equation}
for every  vector $u\in T_p\Sigma$, and the $\theta$-web $v_{[\theta]}$ on the $1/n$-translation surface $\Sigma$.

\begin{defn}
    Let $\Sigma$ be a $1/n$-translation surface  with a compatible Finsler metric $F_Q$. 
    \begin{itemize}
    \item The \emph{length} of a piecewise $C^1$ curve $\gamma:[0,T]\to \Sigma$ is $\ell_Q(\gamma)=\int_0^T F_Q(\gamma'(t))dt$.
    \item    A piecewise $C^1$ curve $\gamma:[0,1]\to\Sigma$ is a \emph{geodesic} if it minimizes length in its homotopy class relative to endpoints.
    \item A closed piecewise $C^1$-curve $\gamma:\mathbb S^1\to\Sigma$ is geodesic if it minimizes length in its free homotopy class.
    \end{itemize}
\end{defn}

As opposed to CAT(0) geodesics (Definition \ref{defn:CAT(0)_geodesic}), the definition of geodesics for general Finsler metrics is not local. This is for a reason: while local minimizers for the CAT(0)-length are necessarily minimizers in their homotopy class, already in $\R^2$ for non-strictly convex Finsler metrics, there exist curves that locally minimize length, but far from minimizing total length.

\begin{ex}
    Consider $\R^2$ with the $\ell^1$ metric. The piecewise straight path $(0,0)-(1,0)-(1,1)-(0,1)$ depicted in blue in Figure \ref{fig:ell1} locally minimizes length but doesn't minimize length in its homotopy class: it has length 3 while the homotopic curve in red has length 1.
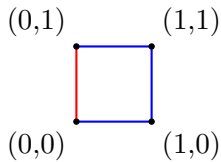
\begin{figure}[h]
\begin{center}
\begin{tikzpicture} 
    \tikzset{every picture/.style={line width=0.75pt}} 
\draw [thick, blue](0,0) to (1,0)to (1,1) to(0,1);
\draw [thick, red](0,0) to(0,1);
\filldraw (0,0) circle (1pt) node [below left] {(0,0)}; 
\filldraw (1,0) circle (1pt) node [below right] {(1,0)}; 
\filldraw (1,1) circle (1pt) node [above right] {(1,1)}; 
\filldraw (0,1) circle (1pt) node [above left] {(0,1)}; 
\end{tikzpicture}
\end{center}
\caption{A local length minimizer for the $\ell^1$-metric in $\R^2$ in blue, which is not a Finsler geodesic.}\label{fig:ell1}
\end{figure}
\end{ex}

We will denote by $d_Q$ the distance function on $\Sigma$ associated to the (compatible) Finsler metric $F_Q$. Observe that by compactness of $\Sigma$ the distance is geodesic, but if the underlying norm $\|\cdot\|_Q$ is not strictly convex, geodesics are in general not unique.

\section{CAT(0) geodesics and Finsler metrics}\label{sec:CAT(0)_geodesic}

In the following, we will use the terms \textit{CAT(0) geodesic} and \textit{Finsler geodesic} to distinguish geodesics with respect to different norms.
 The goal of the subsection is to prove Theorem \ref{thmA} from the introduction, showing that  $\CAT(0)$ geodesics on $1/n$-translation surfaces are  Finsler geodesics for any compatible Finsler metric.

We will work in the universal cover $\wideS$ and define a length non-increasing projection $\pi_w: \widetilde\Sigma\to\gamma$ for any compatible Finsler metric and any $\CAT(0)$ geodesic $\gamma\subset\wideS$.

\subsection{Straight geodesics}
We know from Proposition \ref{geodesic behavior} that geodesics in a $1/n$-translation surface are concatenations of straight segments. Some special geodesics have locally constant tangent vectors. See Figure \ref{p.left-turning} for an illustration. 

\begin{defn}\label{d.left_turning_geodesic}
    A \emph{bi-infinite left-turning} (resp. \emph{right-turning}) \emph{straight $\CAT(0)$ geodesic} (in short \emph{straight geodesic}) in a $1/n$-translation surface $\Sigma$ is an oriented $\CAT(0)$ geodesic $\gamma\subset \Sigma$ such that, for any singular point $p$ in $\gamma$, the angle at $p$ to the left (resp. right) of $\gamma$ is $\pi$.
\end{defn}

\begin{figure}[htbp]
    \centering
    \includegraphics[width=0.35\textwidth]{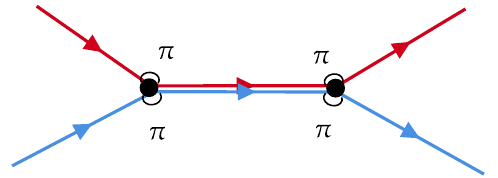}
    \caption{A left-turning straight $\CAT(0)$ geodesic in red, and a right-turning straight $\CAT(0)$ geodesic in blue sharing a saddle connection.}\label{p.left-turning}
\end{figure}

Note that in the definition of straight geodesic the set of singular point in $\gamma$ could be empty, that is, regular geodesics are examples of straight geodesics.

Intuitively, a tangent vector determines a left-turning (resp.~right-turning) straight geodesic. But in general the tangent space at a singular point $y\in P\subset\Sigma$ is not well defined as a vector space and multiplication by scalar $-1$ does not make sense. However a neighbourhood of $y$ can be obtained by glueing open neighbourhoods of the origin in sectors of $\R^2$ bounded by geodesic rays, and thus has a natural piecewise linear structure, making it isometric to a cone of angle $\theta(p)$. We use this structure to talk about the generalized tangent space $T_y\Sigma$ at a singular point $y$, while remembering that, if $y$ is singular, only multiplication by a positive real number is well defined. Note that the orientation of $\Sigma$ also induces a cyclic orientation of $T_p\Sigma$ at singular points.


\begin{prop}\label{p.halfspaces}
    Let $\Sigma$ be a $1/n$-translation surface. Then any vector $v\in T(\Sigma)$ is contained in a unique bi-infinite left- (resp. right-)turning  straight $\CAT(0)$ geodesic. Moreover, if $\Sigma=\wideS$ is simply connected, every bi-infinite left- (resp. right-)turning  straight $\CAT(0)$ geodesic separates $\widetilde \Sigma$ into two convex halfspaces. 
\end{prop}

\begin{proof}
    In order to verify the first claim, observe that straight geodesics are locally uniquely determined by their tangent vectors; since the $\CAT(0)$ distance is complete, local straight geodesic can be extended for all times.
    Suppose one of the half space is not convex, then there is a geodesic with endpoints in this half space and intersecting the other one. As a consequence, there is a bigon enclosed by segments of the boundary and this geodesic, which creates a degenerate geodesic triangle with total inner angle greater than $\pi$. It contradicts Proposition \ref{CAT0 property}.
\end{proof}

Let us now endow $\wideS$ with the CAT(0) metric induced by the $1/n$-translation structure, and consider the Gromov boundary $\partial\wideS$. It is well known that $\partial\wideS$ is homeomorphic to the circle, and the orientation of $\Sigma$ induces the counter-clockwise cyclic orientation on $\partial\wideS$. Note that any oriented geodesic $\gamma$ determines a forward endpoint in $\gamma^+\in\partial\Sigma$ and a backward endpoint  $\gamma^-\in\partial\Sigma$. 

\begin{defn}
Two pairs $(x_1,y_1),(x_2,y_2)\in \partial\widetilde\Sigma\times\partial\widetilde\Sigma$ are \emph{linked} if $x_2,y_2$ are in different connected components of $\partial\wideS\setminus\{x_1,y_1\}$. Two geodesics are said to be \emph{transverse} if their endpoints in $\partial\wideS$ are linked. The notion of \emph{unlinked} pairs and \emph{non-transverse} geodesics are defined analogously.
\end{defn}

For $1/n$-translation surface and $\CAT(0)$ geodesics, transverse and non-transverse pairs of geodesics behave in a special way.

\begin{prop}\label{p.transverse_intersection}
    For a $1/n$-translation surface, if two $\CAT(0)$ geodesics $\gamma_1,\gamma_2$ are transverse, then one of the following holds:
    \begin{enumerate}
        \item $\gamma_1\cap\gamma_2$ is one single point $p$ and the two geodesics intersect transversely at $p$;
        \item $\gamma_1\cap\gamma_2$ is a concatenation of saddle connections, and $\gamma_1$ starts from one side of $\gamma_2$, travels along $\gamma_2$ for those common saddle connections, and then leaves $\gamma_2$ on the other side.
    \end{enumerate}
    If either $\gamma_1$ or $\gamma_2$ is left-turning or right-turning, then only the first case can happen.
\end{prop}

\begin{proof}
    By CAT(0) geometry, the intersection $\gamma_1\cap\gamma_2$ of two geodesics is either empty or a connected subset, otherwise segments of $\gamma_1,\gamma_2$ joining different components of $\gamma_1\cap \gamma_2$ creates a bigon. If the intersection is empty the geodesic $\gamma_2$ lies in one half-space determined by $\gamma_1$, and thus the endpoints cannot be linked. Suppose now $\gamma_1\cap\gamma_2$ is not reduced to a single point, and denote by $p_1,p_2$ the endpoints of the common geodesic segment. Then, by Proposition \ref{geodesic behavior}, the total angle at  $p_i$ is greater than $2\pi$,  and thus the points $p_i$ are cone points. Therefore, $\gamma_1\cap\gamma_2$ is a concatenation of saddle connections.

    In case (2), orient the geodesics $\gamma_i$ so that $p_1$ is before $p_2$ (compare Figure \ref{f.transverse_intersection}). With this orientation we can assume, up to switching the roles of $\gamma_1$ and $\gamma_2$ that the 4-tuple $(\gamma_1^-,\gamma_2^-,\gamma_1^+,\gamma_2^+)$ is positively oriented. With these choices, the angle at $p_2$ to the left of $\gamma_1$ is strictly bigger than the angle at $p_2$ to the left of $\gamma_2$, and is in particular strictly bigger than $\pi$, similarly the angle at $p_1$ to the right of $\gamma_1$ is strictly bigger than $\pi$.  Thus, $\gamma_1$ cannot be left-turning or right-turning, the proof for $\gamma_2$ is analogous.
\end{proof}

\begin{figure}[h]
    \centering
    \includegraphics[width=0.5\linewidth]{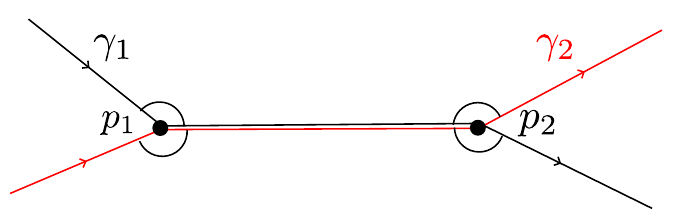}
    \caption{Two transverse geodesics share saddle connections. All the marked angles are no less than $\pi$.}
    \label{f.transverse_intersection}
\end{figure}

\begin{prop}\label{p.unlinked}
    For a $1/n$-translation surface, if two $\CAT(0)$ geodesics $\gamma_1,\gamma_2$ are non-transverse, then one of the following holds:
    \begin{enumerate}
        \item $\gamma_1\cap\gamma_2=\emptyset$;
        \item $\gamma_1\cap\gamma_2$ is a concatenation of saddle connections.
        \end{enumerate}
\end{prop}

 More generally given a tangent vector $w\in T_p\Sigma$ at an either regular or singular point, we denote by $w^+_l$ the forward endpoint of the unique left-turning straight geodesic tangent to $w$, and analogously for $w^-_l, w^+_r, w^-_r$. Denote by $\gamma_l,\gamma_r$ the corresponding left-turning and right-turning geodesics.
 With the same arguments as in the proof of Proposition \ref{p.transverse_intersection} we get
 
\begin{prop}\label{p.endpoint-orientation}\
\begin{enumerate}
\item If $w^+_l$ is distinct to $w^+_r$, and $w^-_l$ is distinct to $w^-_r$, then the four-tuple $(w^+_r, w^+_l, w^-_l, w^-_r)$ is positively oriented. In particular, geodesics $\gamma_l,\gamma_r$ are non-transverse.
\item Let $v_1, v_2, v_3\in T_p\wideS$ be positively oriented, then the triple $({v_1}^+, {v_2}^+, {v_3}^+)$ is positively oriented where $v_i^+$ can be either $(v_i)^+_l$ or $(v_i)^+_r$.
\item Let $I\subset\wideS$ be an oriented regular segment and choose two points $p_1,p_2\in I$ so that $p_2$ comes after $p_1$. Denote also by $I$ the tangent vector along $I$ compatible with the orientation. Let $v_i\in T_{p_i}\wideS$ be a pair of parallel vectors such that $\{I,v_i\}$ form a positively oriented basis. Then $(I_l^+, (v_2)^+_l, (v_1)^+_r, I_l^-)$ is positively oriented. 
\end{enumerate}
\end{prop}
\begin{proof}
It follows from CAT(0) geometry: if the rays were to intersect in an interior point we would obtain either a triangle whose sum of interior angles is bigger than $\pi$, or a bigon, and neither is possible. 
\end{proof}
\subsection{$w$-brushes}
We now introduce $w$-brushes, partial singular foliations of subsets of $\wideS$ determined by a geodesic (segment) $\gamma$ and parallel transverse vectors along $\gamma$.

Let $\gamma:[0,1]\to\widetilde \Sigma$ be a regular geodesic segment. This implies, in particular, that $\gamma$ is contained in an Euclidean neighbourhood, and thus we can talk about parallel transport along $\gamma$. Every  vector $w_0\in T_{\gamma(0)}\wideS$ defines by parallel transport a tangent vector $w_t$ at every point $\gamma(t)$. We will denote by $w$ the corresponding parallel vector field along $\gamma$ and write $w_t=w_{\gamma(t)}$. We say that the vector field $w$ is \emph{positive} if at one interior point (equivalently any point) of $\gamma$, the pair of vectors $\gamma'(t), w_t$ forms a positive basis.

\begin{defn}
    Let $\gamma$ be an oriented regular geodesic segment parametrized by $[0,1]$ and $w$ a positive parallel vector field along $\gamma$. The \emph{$w$-brush} $B_w$ from $\gamma$ is the union of all bi-infinite straight $\CAT(0)$ geodesics tangent to $w$, with the exception that at $t=0$ we only add the right-turning geodesic, at $t=1$ only the left-turning geodesic.
\end{defn}

\begin{figure}[htbp]
    \centering
    \includegraphics[width=0.75\linewidth]{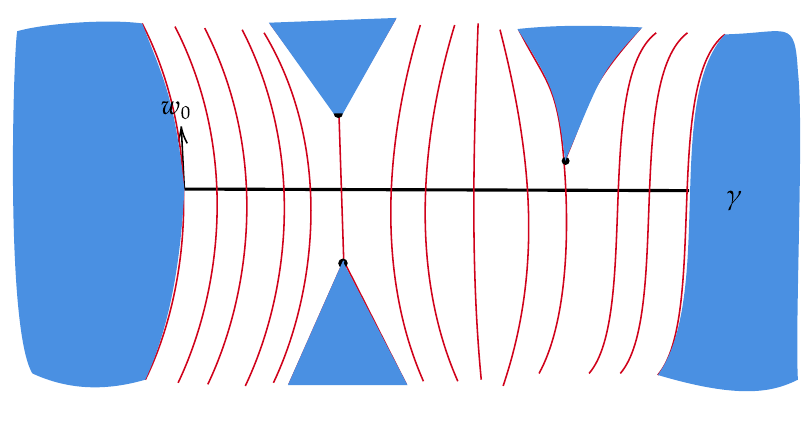}
    \caption{The $w$-brush from the regular geodesic segment $\gamma$, and the blue region is the complement of the brush.}
    \label{f.Brush}
\end{figure}  

We have the following strengthening of Proposition \ref{p.halfspaces}, see Figure \ref{f.Brush} for an illustration:

\begin{prop}\label{p.complementaryregion}
    A $w$-brush from $\gamma$ is a singular foliation of a closed subset of $\widetilde \Sigma$ whose complementary regions are convex and bounded by one of the following:
    \begin{itemize}
        \item two straight geodesics rays meeting at a  cone point,
        \item a straight geodesic ray containing an endpoint of $\gamma$.
    \end{itemize}
    Singularities of the foliation arise precisely at cone points where left- and right-turning geodesics bifurcate.
\end{prop}

\begin{proof}
    In order to show that the $w$-brush foliates its image, it is enough to prove that two straight geodesics intersecting the regular segment $\gamma$ at different points at the same angle are disjoint, which follows from Proposition \ref{p.endpoint-orientation}. It's clear by construction that complementary regions are either half-planes or bounded by two straight geodesic rays. A similar argument to Proposition \ref{p.halfspaces} implies that they are convex.
\end{proof}
Let us now turn to singular geodesics $\gamma$,  the following condition will guarantee that the brushes along the different saddle connections constituting $\gamma$ are disjoint.
\begin{defn}
   Let $\gamma_1\ast\gamma_2$ be a concatenation of oriented saddle connections meeting at a cone point $p$. Let $w^i$ be a positive parallel field along $\gamma_i$, and denote by $w^i_p\in T_p\wideS$ the corresponding vector. The fields $w^1,w^2$ are \emph{unlinked} if the four-tuple $((w^2_p)^+_r,(w^1_p)^+_l,(w^1_p)^-_l,(w^2_p)^-_r)$ is positively oriented.
\end{defn}
Let us now consider an arbitrary geodesic $\gamma$, which we write as concatenation $\gamma=\gamma_1\ast\gamma_2\ldots\ast \gamma_k$ of saddle connections. We denote by $p_i=\gamma_i\cap\gamma_{i+1}$ the cone point where two consecutive saddle connections meet. Furthermore we choose, for every saddle connection a positive vector field $w^i$ along $\gamma_i$.

\begin{prop}\label{prop.unlinkedbrush}
   Let $\gamma=\gamma_1\ast\gamma_2\ldots\ast \gamma_k$ be a geodesic, and $w_i$ positive fields along $\gamma_i$, if for every $i$ the fields $w^i,w^{i+1}$ are unlinked, then for every $i<j$ the brushes $B_{w^i}, B_{w^j}$ are disjoint if $j>i+1$, and intersect in a common segment containing the point $p_i$ (possibly reduced to the point $p_i$) if $j=i+1$.
   \end{prop}

\begin{proof}
   Note that, for each saddle connection $\gamma_i$,  the half-spaces bounded by brush lines are nested. Therefore, it suffices to show that, for every $i$, the brushes $B_{w_i}$ and $B_{w_{i+1}}$  based at the saddle connections $\gamma_i,\gamma_{i+1}$ only intersect in their respective geodesics based at $p_i$. This follows from the hypothesis using Proposition \ref{p.endpoint-orientation}.
%
\end{proof}

    

\subsection{Supporting vectors}
We will be interested in brushes adapted to a Finsler metric $F_Q:T(\widetilde \Sigma\setminus \widetilde P)\to \R$. Recall from Definition \ref{d.cFinslerm} that we assume that $F_Q$ is induced by a norm $\|\cdot\|_Q$ on $\R^2$. In the next definition we choose an identification of $T_x(\Sigma\setminus P)$ with $\mathbb R^2$ induced by the $1/n$-translation structure. Since the norm $\|\cdot\|_Q$ is invariant under $\Z/n\Z$ the definition doesn't depend on the chosen identification.
\begin{defn}

A vector $w$ at a regular geodesic segment $\gamma$ (or equivalently a parallel vector field) is \emph{supporting} for  $F_Q$ if $\norm{\gamma'+tw}_Q\ge\norm{\gamma'}_Q$ for every $t\in \R$.
\end{defn}
\begin{figure}[h]
    \centering
    \includegraphics[width=0.25\linewidth]{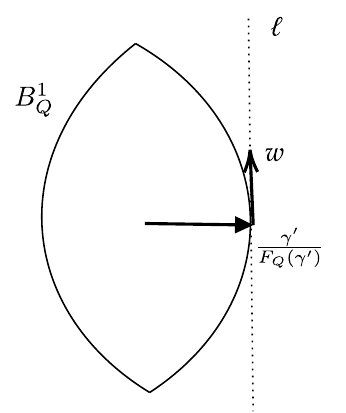}
    \caption{The unit ball of a Finsler norm $F_Q$ and a supporting vector $w$ for a regular geodesic segment.}\label{f.supporting}
\end{figure}

\begin{lemma}
   Any regular geodesic segment $\gamma$ admits a supporting transverse vector.
\end{lemma}

\begin{proof}
 Here we identify the tangent space with $\R^2$ and the unit ball $B^1_Q$ with $Q$. Since the unit ball $Q$ is convex, for any unit vector $v=\gamma'/F_Q(\gamma ')\in Q$ there exists an affine line $l$ containing the point $v$ such that $Q$ is contained in one closed half-space determined by $l$. The vectors tangent to $l$ are supporting for $\gamma'$ (see Picture \ref{f.supporting}). 
 Note that $l$, and thus supporting vectors up to scalar multiples, is unique if and only if $Q$ is $C^1$ at $v$. 
\end{proof}

We will consider in the proof of Theorem \ref{thmA} brush lines associated to supporting vectors.
\begin{prop}\label{p.unlinked}
Let $\gamma=\gamma_1\ast\gamma_2$ be a concatenation of two saddle connections meeting at a cone point $p$. If the positive parallel field $w_i$ is supporting for $\gamma_i$, then $w_1$ and $w_2$ are unlinked.
\end{prop}

\begin{proof}
We orient $\gamma$ so that $\gamma_1$ comes before $\gamma_2$, and denote by $\gamma_1^-$ (resp. $\gamma_2^+$) the backward endpoint of the left-turning geodesic ray tangent to $\gamma_1'$ (resp. the forward endpoint of the right-turning geodesic ray tangent to $\gamma_1'$). It is enough to show that the four-tuple $\big(\gamma_2^+,(w_2)^+_r,(w_1)^+_l,\gamma_1^-\big)$ is positively oriented, or equivalently that the angle at $p$ to the left of $\gamma$ between $\gamma_2'$ and $w_1$ is bigger than that between $\gamma_2'$ and $w_2$. Indeed the same argument, after reversing the orientation gives that $\big(\gamma_1^-,(w_1)^-_l,(w_2)^-_r,\gamma_2^+\big)$ is positively oriented, which shows that $w_1,w_2$ are unlinked.

Note that since $w_i$ are positive supporting fields, since the angle at $p$ to the left of $\gamma$ between $\gamma_2'$ and $w_2$ is smaller than $\pi$, if the angle at $p$ between $\gamma_1'$ and $\gamma_2'$ is at least $\pi$, we are done. It is thus enough to consider the case in which $\gamma_1$ and $\gamma_2$ bound a Euclidean sector with angle less than $2\pi$. 

We choose in this case an identification, induced by the $(G,X)$-structure, with a sector in $\R^2$ and denote by $v_1$ the unit vector of $\gamma_1'$ and by $v_2$ the unit vector of $\gamma_2'$. Note that by assumption the angle between $v_1$ and $v_2$ is smaller than $\pi$, and we can assume, up to rotating the norm $\|\cdot\|_Q$, that $v_2$ is horizontal (compare Figure \ref{fig:projection_supporting_vector}). 

\begin{figure}[h]
    \centering
    \includegraphics[width=0.75\linewidth]{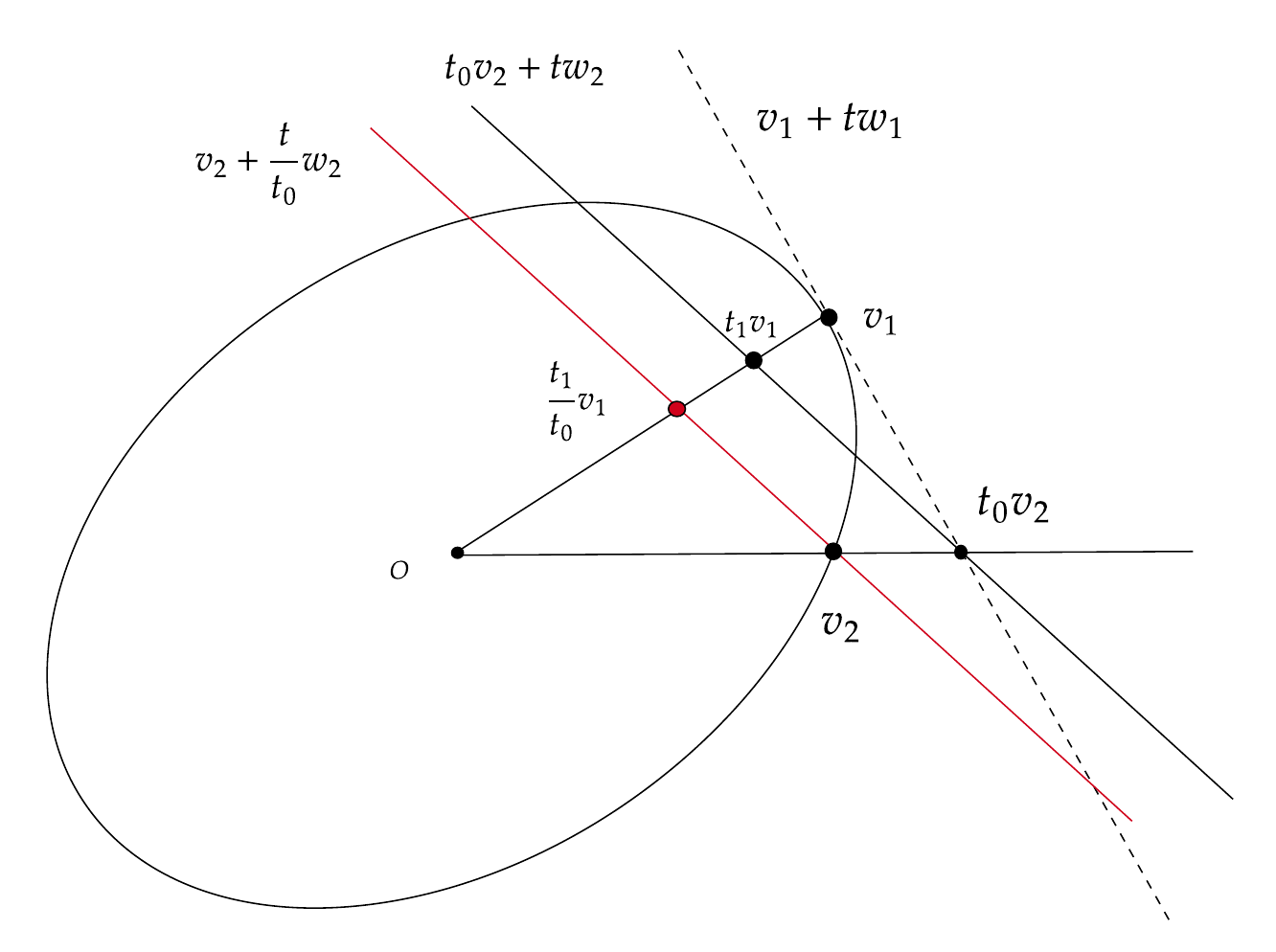}
    \caption{The unit tangent space with respect to a Finsler norm $F_Q$ at the cone point $p$. When the angle between $v_2$ and $w_2$ is greater than the angle between $v_2$ and $w_1$, the brush line at $  v_2$ (the red line) is forced to intersect the interior of the unit ball.}
    \label{fig:projection_supporting_vector}
\end{figure}

We distinguish two cases:
   \begin{itemize}
        \item If the angle of $w_1$ and $v_2$ to the left of $v_2$ is at least $\pi$, then we are done as in the previous case.
        \item Otherwise the affine line 
        $v_1+tw_1$ intersects the ray spanned by $  v_2$ at $t_0  v_2$ for some $t_0\ge1$, see Figure \ref{fig:projection_supporting_vector}. By assumption the angle between $w_1$ and $  v_2$ is smaller than the one of $w_2$. 
        Then the line $t_0  v_2+t  w_2$ will intersect the ray spanned by $  v_1$ at $t_1  v_1$ for some $t_1>0$. Note that both $t_0  v_2+t  w_2$ and $  v_1+t  w_1$ passe through the point $t_0  v_2$. The assumption on the angles of $w_1$ and $w_2$ implies that the line $t_0  v_2+t  w_2$ can be obtained from $  v_1+t  w_1$ by counterclockwise rotation, which forces $t_1<1$. Thus, the point $\frac{t_1}{t_0}  v_1$, which is in the interior of the unit ball as $\frac{t_1}{t_0}<1$, belongs to the supporting line $  v_2+t  w_2$, a contradiction.
    \end{itemize}
\

\end{proof}

\subsection{The projection along a brush and the proof of Theorem \ref{thmA}}

Let now $\gamma$ be any geodesic, and choose for any saddle connection $\gamma_i$ constituting $\gamma$ a positive supporting parallel vector field along $\gamma_i$. Since the vector fields are unlinked (Proposition \ref{p.unlinked}), the brushes $B_{w_i}$ only possibly intersect in their boundaries (Proposition \ref{prop.unlinkedbrush}), and their union is a (mildly) singular foliation $B_w$ of a subset of $\wideS$ whose complementary regions are convex and bounded by two straight geodesic rays meeting at a singular point or at an endpoint of $\gamma$ (Proposition \ref{p.complementaryregion}).

By construction the two straight geodesics containing the two rays bounding a complementary region $\mathcal C$ meet $\gamma$ in the same point, which we denote by $\pi_w(\calC)$. If $\mc C$ is a half-plane, then $\pi_w(\calC)$ is defined to be the corresponding endpoint of $\gamma$. As a result we have:

\begin{defn}\label{brush definition}
    The \emph{projection along a $w$-brush} is the map 
    \[\pi_{w}:\widetilde \Sigma\to \gamma\]
    mapping points in the $w$-brush to the unique intersection of the corresponding straight geodesic with $\gamma$ and points in the complementary region $\calC$ to $\pi_w(\calC)$.
\end{defn}
We then get:
\begin{prop}\label{p.piw}
The projection $\pi_w$ along the $w$-brush is continuous. If the fields $w_i$ are supporting, it is distance non-increasing.
\end{prop}

\begin{proof}
Since the projection is constant on complementary regions, and agrees with its value on their boundary, it is enough to show continuity at points in the $w$-brush. Since in every compact region of $\widetilde \Sigma$ there are only finitely many singular points, for every point $x$ in the $w$-brush not in the boundary of a complementary region, we can find an $\R^2$-chart containing the segment between $x$ and $\pi_w(x)$, foliated by segments in the $w$-brush. In this chart the map $\pi_w$ is simply the restriction of the linear projection of $R^2=\R\gamma'+\R w$ on the first factor, which is clearly continuous. The continuity on points in the boundary of a complementary region follows analogously.
%
%
%
 To prove the second claim, by continuity, it is enough to show that $\pi_w$ is locally distance non-increasing around every  point $x\in\widetilde \Sigma$. 
 We consider three cases.
\begin{itemize}
\item    If $x$ belongs to a complementary region $\calC$, it admits a neighbourhood where $\pi_w$ is constant. 
\item If there is no singular point in the segment in the brush between $x$ and $\pi_w(x)$, then in a suitable $\R^2$-chart the projection $\pi_w$ agrees with the linear projection on the first factor of the splitting $\R^2=\R\gamma'\oplus\R w$. This is norm non-increasing because $w$ is supporting:
    \[\|a\gamma'+bw\|_Q=|a|\|\gamma'+\frac ba w\|_Q\geq |a|\|\gamma'\|_Q=\|a\gamma'\|_Q.\]
 \item   The third case, when $x$ belongs to the boundary of a complementary region follows combining these two cases.
 \end{itemize}
\end{proof}

\begin{proof}[Proof of Theorem \ref{thmA}]
    We will first show that CAT(0) geodesic segments are Finsler geodesics in their homotopy class relative to endpoints. Let $\gamma=\gamma(t),\, t\in[0,1]$ be a $\CAT(0)$ geodesic in $\Sigma$, and $\eta=\eta(t),\, t\in[0,1]$ be another curve homotopic to $\gamma$ relative to endpoints. We choose lifts $\widetilde{\gamma}$ and $\widetilde{\eta}$ in the universal cover $\wideS$ with the same endpoints. Choosing a supporting vector $w$ for $\widetilde \gamma$ with respect to the norm $\|\cdot\|_Q$ inducing the compatible Finsler metric, we have $$\ell_Q(\widetilde\eta)\geq \ell_Q(\pi_w(\widetilde\eta))\geq \ell_Q(\widetilde \gamma),$$
    where the first inequality follows since $\pi_w$ is norm non-increasing, and the second since $\pi_w$ is a continuous projection onto $\til\gamma$ (Proposition \ref{p.piw}).

    Let now $\gamma$ be a closed CAT(0) geodesic,  and $\eta$ be any curve  homotopic to $\gamma$. Choose lifts $\widetilde\gamma,\widetilde\eta: \R\to \wideS$ with the same forward and backward limit points. Note that since $\gamma$ and $\eta$ are homotopic, there is a  deck transformation $\gamma:\wideS\to\wideS$ preserving $\widetilde\gamma$ and $\widetilde\eta$. Note that $\gamma$ is an isometry for both the CAT(0) and the Finsler metric, and preserves the $w$-brush $B_w$ from $\widetilde\gamma$. This implies that if $\widetilde\eta|_{[0,1]}$ is a fundamental domain for the $\gamma$-action on $\widetilde\eta$, then $\pi_w(\widetilde\eta(0))$ and $\pi_w(\widetilde\eta(1))$ must bound a fundamental domain of $\gamma$, and we can apply the same argument as the case of geodesic segments, which finishes the proof.
\end{proof}

\section{Multi-foliations and associated geodesic currents}\label{sec:multifoliation}
The goal of this section is to define multi-foliation currents, the central object in the paper. We first discuss in Section \ref{s.multifol} multi-foliations in a  $1/n$-translation surface $\Sigma$ , in order to construct their associated geodesic currents we then turn to the study in Section \ref{s.mutheta} of small boxes, subsets of $\partial\wideS^2$ that we can understand geometrically, and construct in Section \ref{s.mutheta2} a premeasure defined on such sets.

\subsection{Multi-foliations}\label{s.multifol}
Measured foliations and their associated measured laminations are a key tool in the study of (half-)translation surfaces. For $1/n$-translation surfaces, we will replace them with  \emph{multi-foliations} that we introduce in this section, and for the purpose of construction in Section \ref{s.mutheta2}, we will further give an orientation for each leaf.

In order to obtain an object invariant by the group $\R^2\rtimes\Z/n\Z$, we need to consider collections of $n$-lines through any point (see Figure \ref{f.R2multi}):

\begin{defn}[multi-foliation in $\R^2$]\label{multi-foliation in R}
    For every angle $\theta$, we denote by $\hat{\calF_\theta^n}$ on $\R^2$ the \emph{multi-foliation} given by the union of all \emph{oriented} straight lines in direction $\theta +\frac{2\pi k}n$ for $0\leq k\leq n-1$. The \emph{transverse measure} of a vector $v$ based at the origin is $\sum_{k=0}^{n-1}\abs{\innerproduct{v}{v_{\theta+\frac{2k\pi}{n}+\frac{\pi}{2}}}}=\innerproduct{v}{v_{[\theta+\frac{\pi}{2}]}}$, where $v_{[\theta+\frac{\pi}{2}]}$ is the $n$-web in Definition \ref{web_definition}. 
\end{defn}
\begin{figure}[h]
    \centering
    \includegraphics[width=0.75\linewidth]{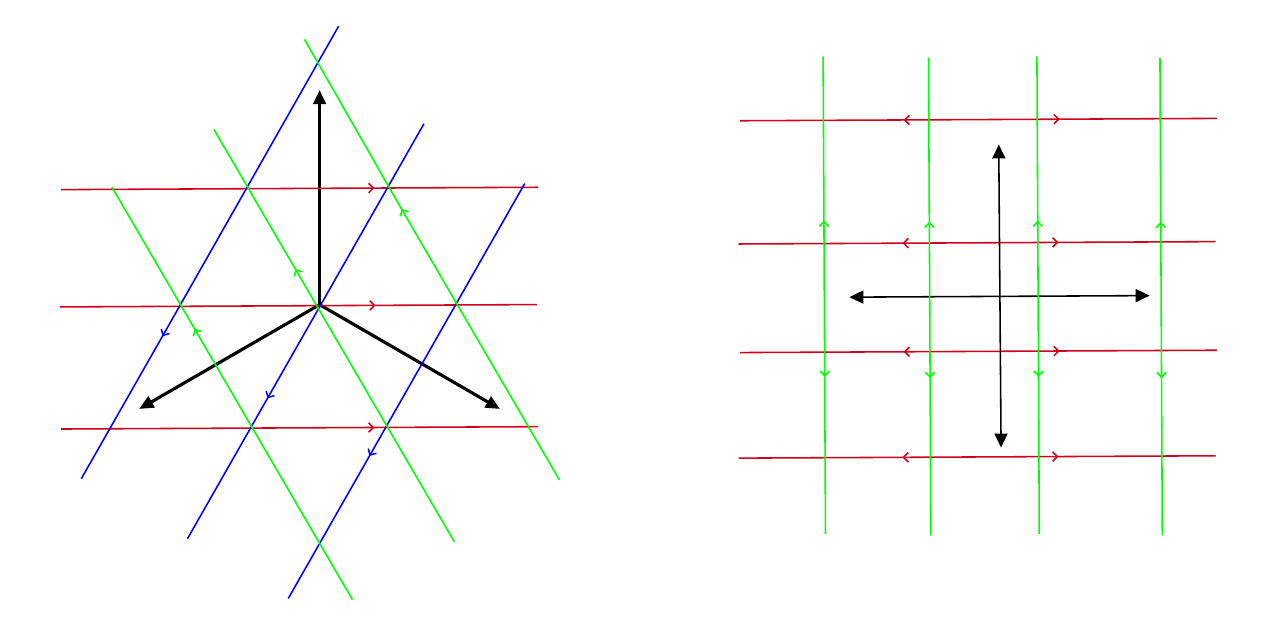}
    \caption{Some horizontal multi-foliations in $\R^2$. The left one is the multi-foliation of a 3-web, and the right one is the multi-foliation for a 4-web.}\label{f.R2multi}
\end{figure}
        
In other words,  the multi-foliation in $\R^2$ is the union of the straight lines orthogonal to one of the vectors in an $n$-web $v_{[\theta+\frac{\pi}{2}]}$, which is oriented so that their tangent and the orthogonal vector in  $v_{[\theta]}$  form a positive basis.

If $n$ is odd, then every point $x\in\R^2$ belongs to $n$ leaves of the multi-foliation, while if $n=2m$ is even, every point belongs to $m$ straight lines but $2m$ oriented lines. In the latter case, the measure is still the sum of $2m$ inner products.

\begin{defn}[multi-foliation on $1/n$-translation surfaces]\label{multi-foliation via local charts}
  Using the $1/n$-translation structure we can pullback, for every angle $\theta$ the multi-foliation $\hat{\calF_\theta^n}$ on $\R^2$ to a well-defined \emph{multi-foliation} $\calF_\theta^n$ on $\Sigma\setminus P$. 
\end{defn}
Of course for every $0\leq k\leq n-1$, the multi-foliation at angle $\theta$ and that at angle $(\theta+\frac{2\pi k}n)$ agree, i.e., $\calF_\theta^n=\calF_{\theta+\frac{2\pi k}n}^n$.
In the following, when there is no confusion, we will just denote the multi-foliation by $\mc F$ or $\mc F_\theta$, depending on whether we want to emphasize the importance of the angle.

In analogy with measured foliations we define:
\begin{defn}\label{d.follength}
The \emph{length function} associated to the  multi-foliation $\mc F_\theta$ on the $1/n$-translation surface $\Sigma$ is the function $\ell_\theta$ such that for every $C^1$-curve $\gamma:[a,b]\to \Sigma$, it holds
\[
\ell_\theta(\gamma)=\int_a^b \innerproduct{v}{v_{[\theta+\frac{\pi}{2}]}} dt.
\]
%
Furthermore, if $\gamma$ is a closed curve, we define the length function of the free homotopy class $[\gamma]$ of $\gamma$ as   the infimum $\ell_{\theta}([\gamma])$ of $\ell_{\theta}(\eta)$ over all closed curves $\eta$ freely homotopic to $\gamma$. 
\end{defn}
We then have

\begin{cor}\label{theta length}
    The length function $\ell_\theta$ is induced by the Finsler metric of the $(\theta+\frac{\pi}{2})$-web on the $1/n$-translation surface $\Sigma$.
\end{cor}



\subsection{Leaves of the multi-foliation and small boxes}\label{s.mutheta}
We will work in this section with the \emph{space of oriented geodesics} 
\[
G^+(\widetilde\Sigma):=(\partial\wideS \times \partial\wideS)\setminus\Delta.
\]
More precisely, since the metric $g$ is CAT(0), the space $G^+(\widetilde\Sigma)$ can be identified with equivalence classes of oriented bi-infinite geodesics $\gamma$ up to parallelism, with the first point being backward limit point $\gamma^-$ and the second point being forward limit point $\gamma^+$. 

We denote by $\widetilde\calF_\theta$ the lift of $\calF_\theta$ to the universal cover $\widetilde \Sigma$, we furthermore drop $\theta$ for the sake of convenience and assume $\theta=0$. 
 Recall from Definition \ref{d.left_turning_geodesic} that a geodesic in $\wideS$ is \emph{straight} if it is left-turning or right-turning.
\begin{defn}
A  \emph{leaf} of the multi-foliation $\widetilde\calF$ is a {bi-infinite straight geodesic}  tangent to $\widetilde\calF$ at all points. We denote by  $G^+(\widetilde\calF)\subset G^+(\wideS)$ the set of oriented pair of endpoints of leaves of $\widetilde\calF$.
\end{defn}
 By Propostion \ref{p.bound_cylinder}, two different leaves $l_1, l_2$ of $\widetilde\calF$ have the same endpoints in $G^+(\wideS)$ if and only if they are contained in a common cylinder.

\begin{lemma}\label{l.suppclosed}
The set $G^+(\widetilde\calF)$ is closed.
\end{lemma}
\begin{proof}
Equip the surface with the CAT(0) metric coming from the $1/n$-translation structure. Let $\mc G$ be the space of bi-infinite geodesics in $\wideS$. According to \cite[Proposition 2.2]{bankovic2018marked}, the map sending a bi-infinite geodesic to its endpoints is a closed map. The same proof applies to the map sending a geodesic to oriented pair of endpoints. So it is enough to show that the set of leaves forms a closed subset of the set of all bi-infinite geodesics. Let $l_n$ be a sequence of leaves converging to a bi-infinite geodesic $l$. Choose a regular point $p\in l$ and an Euclidean neighborhood $U$ of $p$. Then $l_n\cap U$ are parallel straight lines, so $l$ is  locally a straight line parallel to them. Thus, the limit is tangent to $\widetilde\calF$. If $l_n$ accumulates to $l$ on its left (resp. right), then at any singular point $q\in l$ the leaf $l$ must  make an angle $\pi$ on the left (resp. right).
\end{proof}

We now discuss the structure of $G^+(\widetilde \calF)$. We would like to construct for each point in $G^+(\widetilde\calF)$ a neighborhood such that all leaves with endpoints in this neighborhood are``parallel". See Figure \ref{orthogonal segment} for an illustration. 

Let us first fix some notation. If $A$ is a subset of $G^+(\wideS)$, we denote by $L(A)$ the set of bi-infinite straight leaves $l$ in $\wideF$ with endpoints $(\ell^-,\ell^+)\in A$. We use \emph{interval} for a proper connected subset of $\partial\wideS$.

\begin{prop}\label{define measure}
For every pair $(z,w)\in G^+(\widetilde\calF)$ there exists a regular segment $I\subset \widetilde\Sigma$ and open 
intervals $C_z$ and $C_w\subset\partial\wS$ 
such that $z\in C_z, w\in C_w$ and every leaf in $L(C_z\times C_w)$ meets $I$ orthogonally. 
\end{prop}
\begin{figure}[h]
    \centering
    \includegraphics[width=\linewidth]{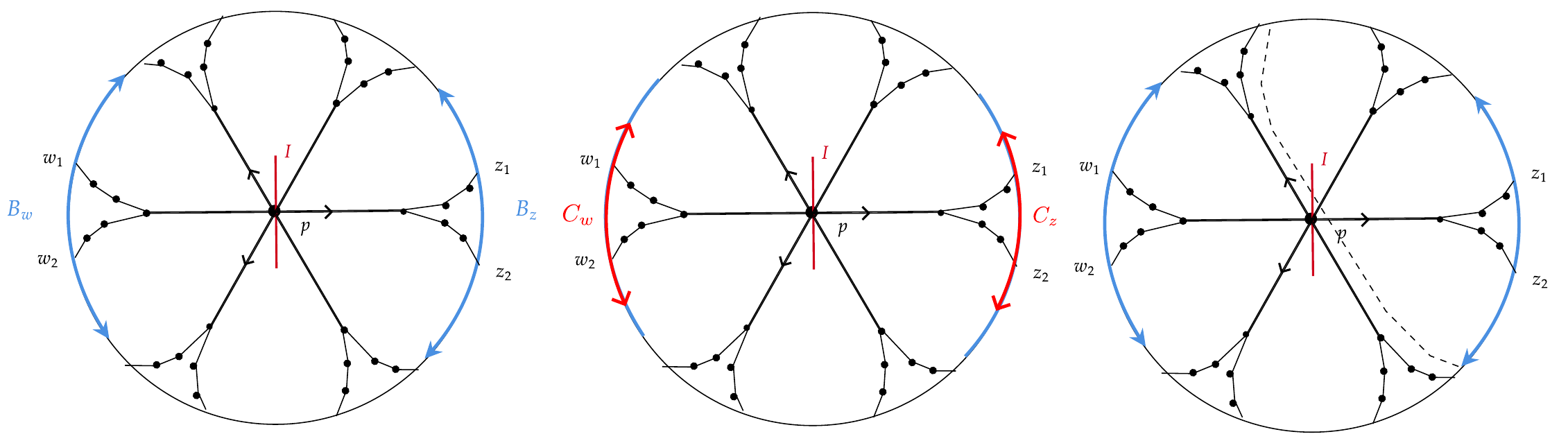}
    \caption{Construction of the neighborhoods in the case of $\frac{1}{3}$ multi-foliation. The picture on the left illustrates all leaves passing through $p$ (left-turning and right-turning) and the intervals $B_z,B_w$ in the boundary. The picture in the middle illustrates the intervals $C_z, C_w$. The picture on the right shows a leaf $l'$ which intersects $I$ in direction $\frac{2\pi}{3}$. It is locally parallel to $l_1$ in a neighborhood of $I$ and its endpoins lie in the semi-circle bounded by $l_1$.}
    \label{orthogonal segment}
\end{figure}
\begin{proof}
     Let $\widetilde{l}_{zw}\subset\widetilde\Sigma$ be a leaf with endpoints $(z,w)$ and $l_{zw}\subset\Sigma$ be its projection. Let $J$ be a regular geodesic arc  orthogonal to $\widetilde{l}_{zw}$ whose endpoints do not belong to regular cylinder leaves (we can always assume this up to possibly enlarging $I$ past a cylinder), and let $I$ be a regular geodesic arc containing $J$ in its interior. 

    We first show that there exist open neighborhoods $A_z,A_w\subset \partial\wS$ of $z$ and $w$ such that every leaf in $L(A_z\times A_w)$  intersects $I$ transversely. Suppose there are no such $A_z,A_w$. Then there exists a sequence of pairs $(z_n,w_n)$ converging to $(z,w)$ such that there exist leaves $\wt{l}_{z_nw_n}$ with endpoints $(z_n,w_n)$ not intersecting $I$. By \cite[\S 4.14]{ballmann2006manifolds}, there exists a compact set $K\subset \wideS$ such that for every $n$, the leaf $\wt{l}_{z_nw_n}$ intersects $K$. Thus, there exists a convergent subsequence of $\wt{l}_{z_nw_n}$ whose limit has endpoints $(z,w)$. Since, by construction, any leaf with endpoints $(z,w)$ intersects $J$ and intersecting the interior of $I$ transversely is an open condition, we deduce that for for $n$ big enough $\wt{l}_{z_nw_n}$ intersects $I$, a contraddiction.

    Next we will show that, up to shrinking $A_z\times A_w$, every leaf intersects $I$ orthogonally. Recall that we assume that $\wideF$ is horizontal. Let $p=\widetilde {l}_{zw}\cap I$, then by the choice of $I$, $p$ is a regular point. Consider the bi-infinite straight leaves of $\wideF$ passing through $p$. Each direction gives rise to at most two leaves, namely one left-turning and one right-turning. In particular, there are only finitely many such leaves. Denote by $(z_1,w_1),(z_2,w_2)$ the endpoints of the leaves through $p$ in  the same direction as $\widetilde{l}_{zw}$: depending on $p$ it might be that $z_1=z_2$ or $w_1=w_2$. By construction the pair $(z,w)$ is either $(z_1,w_1)$ or $(z_2,w_2)$. The endpoints of leaves passing through $p$, which are finitely many points, separate $\partial\wideS$ into disjoint open intervals. Let $B_z$ be the union of $[z_2,z_1]$ with its two neighboring intervals. The open intervals $B_w, B_z$ are pictured in the left of Figure \ref{orthogonal segment}. Define $C_z=A_z\cap B_z$. Similarly we construct $B_w$ and $C_w$. 

    We claim that every leaf with endpoints in $C_z\times C_w$ intersects $I$ orthogonally. Suppose by contradiction that there is a leaf $l'$ that intersects $I$ but not orthogonally. Since $I$ is a regular segment, it lies in an Euclidean neighborhood. The leaves are locally of the form $\mc F_{\frac{2k\pi}{n}}$ and $l'$ intersects $I$ in directions $\frac{2k\pi}{n}$ for some $k\neq 0$. By Proposition \ref{p.endpoint-orientation}, both endpoints of $l'$ lie in the same half-circle bounded by the endpoints of $l_k$. By construction the endpoints of $l'$ cannot be in $B_z\times B_w$ and thus not in $C_z\times C_w$. See Figure \ref{orthogonal segment}.
\end{proof}

\begin{defn}\label{d.small_box}
    A \emph{small box} is a subset of $G^+(\wideS)$ of the form $C\times D$ where $C,D\subset \partial \wS$ are pre-compact intervals, and there exists a closed regular geodesic segment $I\subset \widetilde \Sigma$  such that every leaf in $L(C\times D)$ intersects $I$ orthogonally. 
    \end{defn}
Note that $L(C\times D)$  might be empty, and if $I$ intersects a leaf in a cylinder, then it must intersect all leaves in the same cylinder. 
We record here two useful consequences of the construction in Proposition \ref{define measure}:
\begin{cor}\label{parallel}
Let $C\times D\subset G^+(\wideS)$ be a small box. Then two leaves in $L(C\times D)$ are either disjoint or overlap along a saddle connection tangent to $\widetilde\calF$.
\end{cor}

\begin{proof}
    The leaves in $L(C\times D)$ are straight geodesics orthogonal to the same interval, thus they cannot have transverse intersection. The result follows from Proposition \ref{p.halfspaces}.
\end{proof}

\begin{cor}
Open small boxes generate the topology of $G^+(\wideS)$.
\end{cor}
\begin{proof}
Note that if $C\times D$ is a small box, then for every intervals $C_1\subset C$, $D_1\subset D$, the product $C_1\times D_1$ is also a small box. Since $G^+(\widetilde\calF)$ is closed by Lemma \ref{l.suppclosed}. If $(x,y)$ does not belong to $G^+(\widetilde\calF)$, then $(x,y)$ is contained in arbitrary small boxes outside $G^+(\widetilde\calF)$, while if $(x,y)\in G^+(\widetilde\calF)$, it is contained in arbitrarily small boxes by Proposition \ref{define measure}. 
\end{proof}

We finish the section by discussing an useful converse of Proposition \ref{define measure}: any sufficiently small interval determines a small box.
We begin with a preliminary observation, illustrated in  Figure \ref{leaf intersection connected}.

\begin{prop}\label{l.Iclosed}
  Let $C\times D$ be a small box, and $I\subset \widetilde\Sigma$ be an associated segment. The intersection   $I\cap L(C\times D)$ is  a connected sub-segment $I_0$ of $I$, and any bi-infinite straight leaf of $\calF$ intersecting orthogonally the interior of $I_0$ belongs to $L(C\times D)$.
\end{prop}

\begin{figure}[h]
    \centering
    \includegraphics[width=0.35\linewidth]{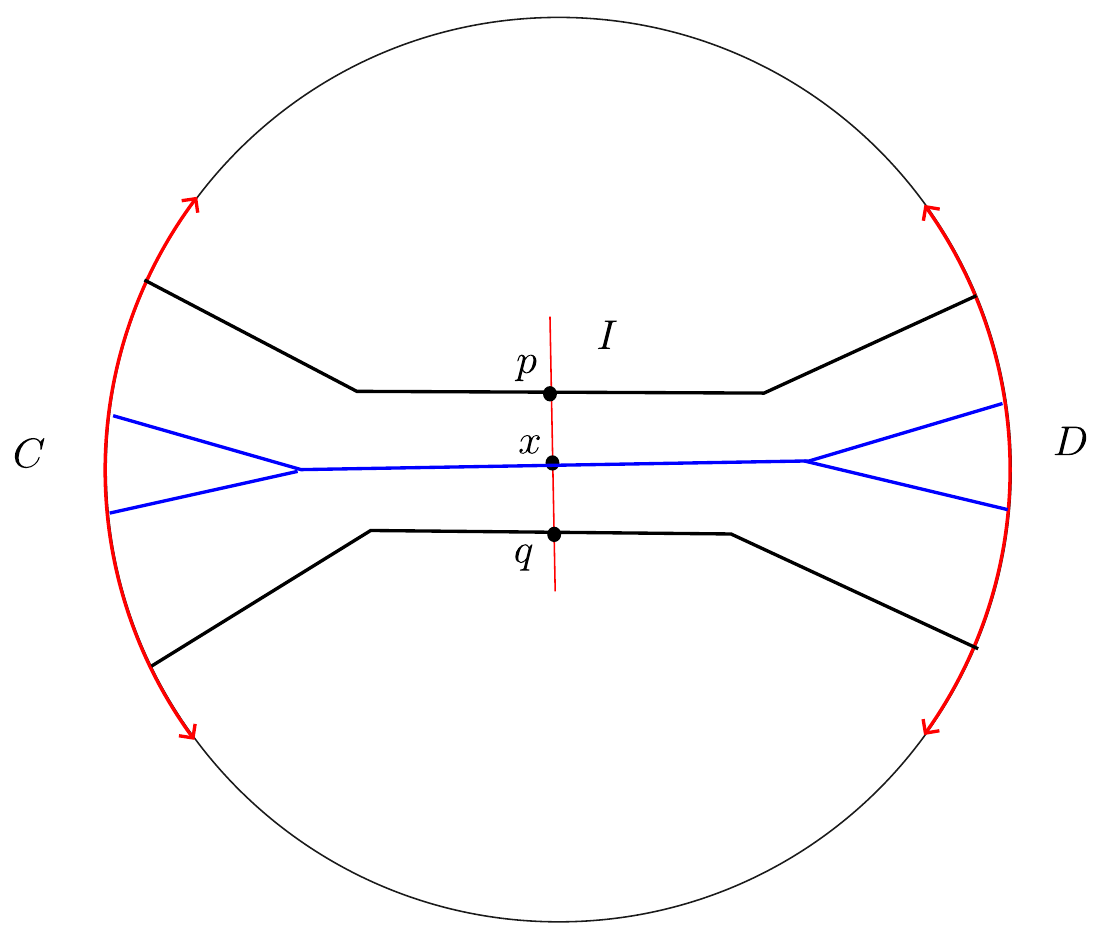}
    \caption{The segment $I\cap L(C\times D)$ is connected. Blue lines are leaves passing through $x$.}
    \label{leaf intersection connected}
\end{figure}

\begin{proof}
Let $p,q \in I\cap L(C\times D)$ and $x\in I$ between $p,q$; since $x$ is regular, there are either one or two leaves of $\calF$ containing $x$ and orthogonal to $I$.
The endpoints of either leaf  lies in the intervals cut out by endpoints of leaves passing through $p$ or $q$. Thus, the leaves are in $L(C\times D)$ and $x\in I\cap L(C\times D)$.
\end{proof}

Each small box is associated with a regular orthogonal segment, which is not unique as the segment can always be extended. Conversely, we want to construct a canonical small box for each short enough regular segment. Let now $I=\overline{x_1x_2}\subset\wideS$ be a regular segment positively orthogonal to a leaf $l$ of $\wideF$. From the local structure of $\wt\calF$, we know there exists a family of leaves through $I$ orthogonally. Let $l_1$ be the left turning leaf of $\wt\calF$ through $x_1$ and $l_2$ be the right-turning leaf of $\wt\calF$ through $x_2$. We denote by $C_I\subset\partial\wideS$ the closed interval $[l_2^-,l_1^-]$ and by $D_I$ the interval $[l_1^+,l_2^+]$. Note that the intervals $C_I$, $D_I$ are degenerate precisely when $l_1,l_2$ are contained in a cylinder.
We want to make precise when the set $C_I\times D_I$ is a small box. See Figure \ref{f.smallbox_construction}.

\begin{defn}\label{d.smallinterval}
A \emph{small interval} $I\subset\wideS$ is a regular geodesic segment such that $C_I\times D_I$ is a small box and $I$ intersects all leaves in $L(C\times D)$ orthogonally. 
\end{defn}


\begin{prop}\label{l.smallbox_cylinder}
    For any cylinder $C\subset\wideS$ tangent to $\til\calF$ every interval $I$ with endpoints in the boundary leaves of $C$ and orthogonal to the leaves in the cylinder is a small interval. In this case 
    $C_I\times D_I=\{l^-\}\times \{l^+\}$ where $l$ is any leaf in the cylinder. 
\end{prop}

\begin{proof}
   Indeed in this case $L(C_I\times D_I)$ consists of all the leaves in the cylinder, and by construction they intersect $I$ orthogonally.
\end{proof}

\begin{lemma}\label{l.finite_cylinder}
    There are only finitely many maximal cylinders in $\Sigma$ foliated by leaves of $\mc F$.
\end{lemma}

\begin{proof}
    By Proposition \ref{p.closed_conepoint} the boundary of a cylinder foliated by leaves of $\mc F$ is a concatenation of saddle connections tangent to $\calF$. Since there are finitely many cone points, and for any cone point $p$ of cone angle $\frac{2k\pi}{n}$  there are at most $k$ saddle connections tangent to $\mc F$ with endpoint $p$, there are at most finitely many such saddle connections. Each saddle connection appears in the boundary of at most two maximal cylinders, which concludes the proof.
\end{proof}
\begin{defn}\label{d.boxing}
An Euclidean rectangle $R\subset\wideS$ is a \emph{boxing rectangle} if 
\begin{enumerate}
\item its horizontal sides are tangent to $\widetilde\calF$, \item its diagonals make an angle with the horizontal sides smaller or equal to $2\pi/n$, 
\item neither horizontal side of $R$ is contained in a regular cylinder leaf. 
\end{enumerate}
\end{defn}
 The following proposition explains the terminology:
\begin{prop}\label{l.smallbox_noncylinder}
Assume that the regular interval  $I$ is a vertical side of a boxing rectangle, then $I$ is a small interval.
\end{prop}

\begin{figure}[h]
    \centering
    \includegraphics[width=0.5\textwidth]{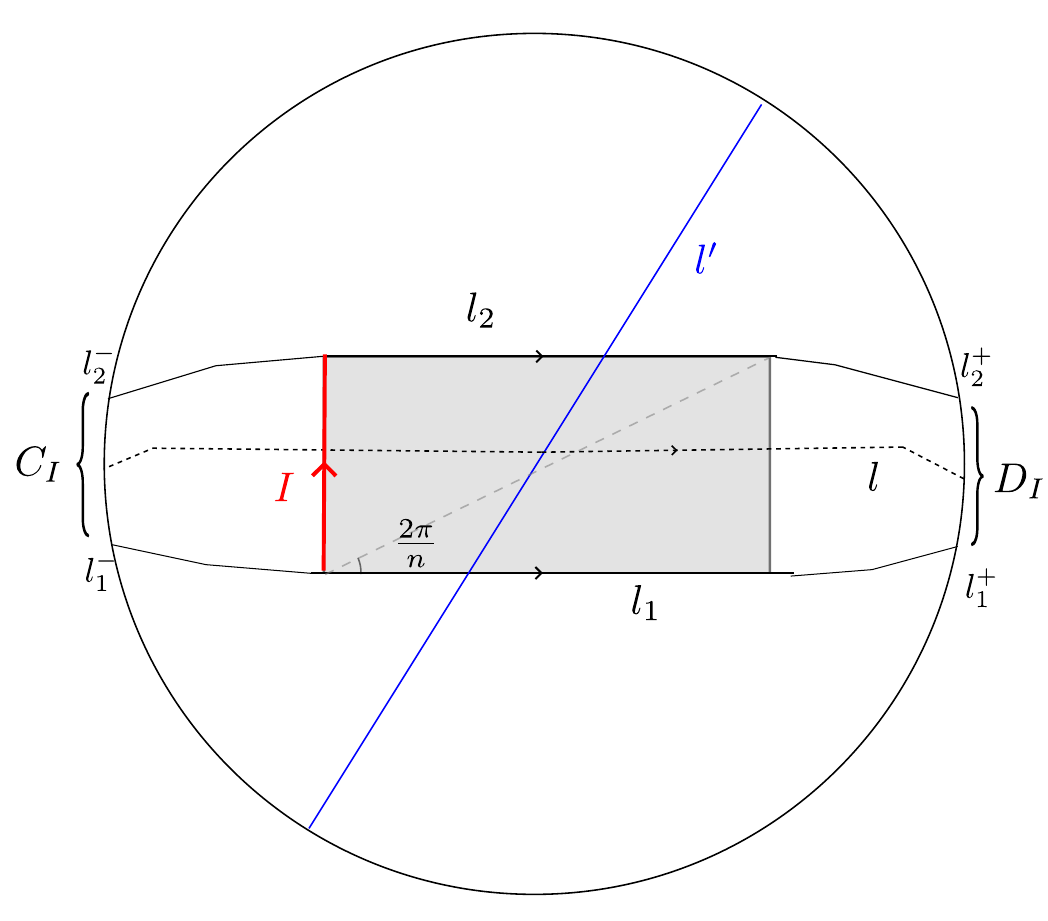}
    \caption{Construction of the small box. Two boundary leaves are left-turning and right-turning respectively. Any leaf $l'$ of $\wideF$ intersecting $I$ not orthogonally have endpoints outside of $C_I\times D_I$.}
    \label{f.smallbox_construction}
\end{figure}


\begin{proof}
Since neither horizontal side of the rectangle is contained in a regular cylinder leaf, if a cylinder leaf intersects $I$ orthogonally, then the entire cylinder intersects $I$ orthogonally. We need to show that any leaf  with endpoints in $C_I\times D_I$ intersects $I$ orthogonally. If by contradiction, a leaf $l'\in L(C_I\times D_I)$ were to meet $I$ not orthogonally, in the $\R^2$-chart in which $I$ is vertical it would correspond to an affine line of angle $2\pi k/n$ for some $1\leq k<n$. The second assumption on $R$ implies that $l$ would  intersect transversely one of the two horizontal sides of $R$, which contradicts the assumption that $l\in L(C_I\times D_I)$ by Proposition \ref{p.endpoint-orientation} (2).  
%
%
\end{proof}

\begin{rmk}
Note that if, instead, the geodesic segment $I$ is too long there can very well be leaves in $C_I\times D_I$ that intersect $I$ not orthogonally. 
\end{rmk}
\subsection{The multi-foliation current $\mu_\theta$}\label{s.mutheta2}
We now construct of the multi-foliation current $\mu_\theta$, a geodesic current on $\Sigma$ that encodes the length function of the multi-foliation $\mathcal F_\theta$. This parallels the construction of the measured lamination associated to the horizontal foliation of a quadratic differential, see for instance \cite[Section 11.3]{kapovich2001hyperbolic}.
\begin{defn}
    An \emph{oriented geodesic current} is a $\pi_1(\Sigma)$-invariant, locally finite, Borel measure on $G^+(\wideS)$. 
\end{defn}

The first step in the construction is to define the measure of small boxes. 
\begin{prop}
Let $C\times D$ be a small box and $I, J\subset \widetilde\Sigma$ be  regular geodesic segments intersecting  every leaf in $L(C\times D)$ orthogonally. Then $|I\cap L(C\times D)|=|J\cap L(C\times D)|$.  
\end{prop}
\begin{proof}
The interior  $I_0$ of $I\cap L(C\times D)$ and the interior $J_0$ of $J\cap L(C\times D)$ are regular geodesic segments (Proposition \ref{l.Iclosed}); we will show that they bound, together with two geodesics tangent to $\widetilde\calF$ an Euclidean rectangle. 

 By Proposition \ref{l.Iclosed}, any leaf of $\wideF$ containing a point $x\in I_0$ and orthogonal to $I$ belongs to $L(C\times D)$ and thus, by construction, also intersects the regular interval $J$. As a result the function $f^l:I_0\to J$ that to a point $x\in I_0$ associates the intersection of the left-turning geodesic from $x$ with $J_0$ is monotone and bijective, and since the same holds for the function $f^r:I_0\to J$ that associates to a point $x\in I_0$  the intersection of the right-turning geodesic from $x$ with $J$, the two functions agree and give continuous bijections between $I_0$ and $J_0$.
 
We can  assume up to switching the roles of $I$ and $J$  that the leaves of $\wideF$ hit $I$ before $J$. We denote by $p,q$ the endpoints of $I_0$, chosen so that $I_0$ is to the left of $p$ and to the right of $q$. Then the left-turning leaf of $\wideF$ from $p$ meets $J$ in the endpoint $f^l(p)$  of $J_0$, and similarly if the right-turning leaf of $\wideF$ from $q$ meets $J$ in the endpoint $f^r(q)$ of $J_0$. Denoting by $L$ the interval between $p$ and $f^l(p)$ and by $R$ the interval between  $q$ and $f^r(q)$, we have that the CAT(0) geodesics $L, J_0, R, I_0$ meet with a right angle; this concludes the proof.
\end{proof}

 We  can now define the function $\mu_\theta$ on small boxes associated to the multi-foliation $\calF_\theta$ by \[
\mu_\theta(C\times D)=
\left\{\begin{array}{ll}
    |I\cap L(C\times D)|,& \text{if $L(C\times D)$ is non-empty;} \\
    0,& \text{otherwise.}
\end{array}
\right.
\]
Here $I$ is any interval that meets all leaves in $L(C\times D)$ orthogonally.  

\begin{lemma}
The function $\mu_\theta$ is invariant under $\pi_1(\Sigma)$: 
\end{lemma}
\begin{proof}
Of course for every $g\in\pi_1(\Sigma)$, and every pair of disjoint intervals $C,D\subset\partial\Sigma$ the set $L(C\times D)$ is non-empty if and only if $L(g.C\times g.D)$ is. Furthermore for every small box $C\times D$, if an interval $I\subset \wS$ meets orthogonally every leaf in $L(C\times D)$, then for every  $g\in\pi_1(\Sigma)$ the interval $g.I$ meets orthogonally every leaf in $L(g.C\times g.D)$, and, as $g$ acts on $\wS$ by isometries, we have $|I\cap L(C\times D)|=|g.I\cap L(g.C\times g.D)|$.
\end{proof}

Next we want to use Carath\'{e}odory's theorem \ref{t.Caratheodory} to extend the function $\mu_\theta$ to a measure on the Borel $\sigma$-algebra. A good reference for the measure theory preliminaries we need is \cite{klenke2013probability}.


\begin{defn}
    A family $\mc R$ of subsets of a set $\Omega$ is a $\textit{ring of sets}$  if it satisfies
    \begin{enumerate}
        \item $\emptyset\in\mc R$;
        \item $\forall A,B\in\mc R$, $A\setminus B\in\mc R$;
        \item $\forall A,B\in\mc R$, $A\cup B\in\mc R$.
    \end{enumerate}
\end{defn}
We denote by $\mc S_\theta$ the family of finite disjoint union of small boxes.

\begin{prop}
The family  $\mc S_\theta$ is a ring.
\end{prop}
\begin{proof}
1) By construction the empty set belongs to $\mc S_\theta$.

2) Let $A=C_1\times C_2,\, B=D_1\times D_2$ be two disjoint small boxes. Then
  \[
 (C_1\times C_2)\setminus (D_1\times D_2)=K_1\sqcup K_2\]
where $K_1,K_2$ are the small boxes
        \[\begin{array}{l}
        K_1=(C_1\setminus D_1)\times C_2,\\
        K_2=(C_1\cap D_1)\times (C_2\setminus D_2).
        \end{array}.\]

3)Since $\mc S_\theta$ is invariant under disjoint finite unions,  
$$A\cup B=(A\setminus B)\sqcup B \in\mc S_\theta.$$
    
    \end{proof}

The following lemma will be useful to extend the function $\mu_\theta$ to a pre-measure on the ring $\mathcal S_\theta$ of finite disjoint unions of small boxes:
\begin{lemma}\label{l.sigmaadd}
Let $A=C\times D$ be a small box. If there are countably many disjoint small boxes $A_i=C_i\times D_i$ such that $A=\sqcup A_i$, then 
$$\mu_\theta (C\times D)=\sum_i\mu_\theta (C_i\times D_i).$$
\end{lemma}
\begin{proof}
The proof of this fact follows along the lines of the proof of Proposition \ref{define measure}: Let $I$ be a regular segment meeting all leaves in $L(C\times D)$ orthogonally, and similarly, for every $i$, let $J_i$ be a regular segment meeting all leaves in $L(C_i\times D_i)$ orthogonally. Then we have bijections between the interiors of the intervals $J_i$ with subintervals $I_i$ of $I$. Then the open intervals $I_i$ are pairwise disjoint and the union of their closures is $I$. As a result the length of $I$ is the sum of the lengths of the intervals $I_i$.
\end{proof}
\begin{cor}
The linear extension of $\mu_\theta$ to $\mathcal S_\theta$ does not depend on the decomposition of an element of $\mathcal S_\theta$ as a disjoint union of small boxes. 
\end{cor}
The extension $\mu_\theta:\mc S_\theta\to \R$ is then naturally $\pi_1(\Sigma)$-invariant.

\begin{defn}
    Let $\mc R$ be a ring of subsets of $\Omega$. A function $\mu:\mc R\to [0,\infty]$  is  a \textit{pre-measure} if 
    \begin{itemize}
        \item $\mu(\emptyset)=0$;
        \item whenever $A=\sqcup_1^\infty A_i,A,A_i\in \mc R$ and $A_i$ pairwise disjoint, then $\mu(A)=\sum_1^\infty \mu(A_i)$.
    \end{itemize}
  The function  $\mu$ is $\sigma$\textit{-finite} if there exist elements $\Omega_n\in\mc R$ such that $\mu(\Omega_n)<\infty$ and $\Omega=\cup_n\Omega_n$. 
\end{defn}

\begin{prop}
 The function   $\mu_\theta$ is a $\sigma$-finite pre-measure on $\mc S_\theta$.
\end{prop}

\begin{proof}
    It is clear that $\mu_\theta(\emptyset)=0$, furthermore the function $\mu_\theta$ is $\sigma$-additive by Lemma \ref{l.sigmaadd}.
    
    In order to show that $\mu_\theta$ is $\sigma$-finite, since $\wideS$ admits an exhaustion of compact sets, it is enough to verify that every point $x\in\wideS$ has an open neighbourhood $U_x$ such that the set $G^+(U_x)$ of endpoints of leaves of $\wideF$ intersecting $U_x$ can be covered by finitely many small boxes. We distinguish two cases:
    
 If $x$ is regular, then there exists a small ball $U_x$ centered in $x$ that does not contain any singular point. Let $I_1\ldots, I_n$ be the maximal regular segments contained in $U_x$ tangent to $v_{[\theta+\pi/2]}$. Every leaf $l\in G^+(U_x)$ intersects orthogonally the interval $I_i$ for some $1\leq i\leq n$. By extending  $I_i$ after possibly sliding it along the foliation  $\wideF$ to  guarantee that $I_i$ crosses the maximal cylinders it intersect and avoids cone points. We obtain $n$ small boxes covering $G^+(U_x)$.

    If $x$ is a cone point of cone angle $\frac{2\pi p}{n}$, there exists a local neighborhood $U_x$ that does not contain other singular points and can be divided into $2p$ sectors separated  by rays of $\wideF_{\theta+\frac{\pi}{2}}$ starting from $x$ of  angle $\frac{\pi}{n}$ at $x$. The union of $n$ consecutive sectors is a sector of angle $\pi$ at $x$. Each leaf of the foliation $\wideF_{\theta+\frac{\pi}{2}}$ contained in $V$ lies in one of the $2p$ $\pi$-sectors and is parallel to its boundary side, it does then belong to a small box associated to the $\pi$-sector and constructed as in the regular case.
\end{proof}

Recall that a \emph{$\sigma$-algebra} on a set $\Omega$ is a family $\mc A$ of subsets of $\Omega$ that contains the empty set and is closed under complements and countable unions.   Given a ring $\mc R$, the \emph{$\sigma$-algebra generated by $\mc R$} is the smallest $\sigma$-algebra $\sigma(\mc R)$ that contains $\mc R$.     If $\Omega$ is a topological space, then the \textit{Borel} $\sigma$\textit{-algebra} is the smallest $\sigma$-algebra generated by all open subsets.

Observe that since  the set of open small boxes is a basis for the topology of $G(\wideS)$, the $\sigma$-algebra generated by $\mathcal S_\theta$ is the Borel $\sigma$-algebra.

\begin{thm}[Carath\'eodory's extension theorem, see e.g. {\cite[Theorem 1.41]{klenke2013probability}}]\label{t.Caratheodory}
    Let $\mc R$ be a ring of sets on $\Omega$ and let $\mu:\mc R\to[0,\infty]$ be a $\sigma$-finite pre-measure. Let $\sigma(\mc R)$ be the $\sigma$-algebra generated by $\mc R$. Then there exists a unique measure $\mu':\sigma(\mc R)\to[0,\infty]$ such that $\mu'$ is an extension of $\mu$.
\end{thm}

With a slight abuse of notation we  denote by $\mu_\theta$ also the Caratheodory extension of the $\sigma$-finite pre-measure $\mu_\theta$ on $\mc S_\theta$ to the Borel $\sigma$-algebra.
By construction $\mu_\theta$ is a $\pi_1(\Sigma)$-invariant and locally finite measure on $G^+(\wideS)$, namely an oriented geodesic current. We will refer to it as the \emph{multi-foliation current}.

The following two easy consequences of the construction play an important role in the study of the multi-foliation current.
\begin{prop}\label{p.support}
The support of the multi-foliation current $\mu_\theta$ is  $G^+(\wideF)$
\end{prop}
\begin{proof}
The set $G^+(\widetilde{\mc F})$ of endpoints of straight leaves of the foliation $\wideF$ is closed by Lemma \ref{l.suppclosed}. In particular every pair $(z,w)\in G^+(\wideS)\setminus G^+(\widetilde{\mc F})$ is contained in a small box $A$ that does not contain leaves of $\wideF$. Since by construction for such $A$ it holds $\mu_\theta(A)=0$, it follows that $\supp{\mu_\theta}\subset G^+(\widetilde{\mc F})$. 

Conversely, if $(z,w)\in G^+(\wideS)$ corresponds to a cylinder, then it corresponds to an atom of the current, and in particular belongs to the support. Otherwise consider the leaf $l\subset\wideS$ corresponding to $(z,w)$ and choose a regular point $p\in l$ and a sufficiently small interval $I$ satisfying the assumption of Proposition \ref{l.smallbox_noncylinder}. Since there are finitely many cylinders in the direction of $\widetilde\calF$ in $\Sigma$ (Lemma  \ref{l.finite_cylinder}), there is a lower bound on the distance between boundary leaves of such cylinders, and in particular we can assume, up to shrinking $I$ that none of the leaves intersecting $I$ orthogonally belong to a cylinder. This implies that any subinterval $I_k$ of $I$ is small, and as $I_k$ converges to $\{p\}$ the small boxes $C_{I_k}\times D_{I_k}$ give rise to arbitrarily small neighborhoods of $(z,w)$ of positive $\mu_\theta$-measure. This implies that $\supp{\mu_\theta}=G^+(\widetilde{\mc F})$.
\end{proof}

\begin{prop}\label{l.atomic}
    The atomic part of $\mu_\theta$ is exactly the set of endpoints of cylinder leaves of $\wideF$. 
\end{prop}

\begin{proof}
    It follows from Proposition \ref{l.smallbox_cylinder} that endpoints of cylinder leaves is an atom of $\mu_\theta$. If $(z,w)$ corrresponds to endpoints of a non-cylinder leaf, then by Proposition \ref{l.smallbox_noncylinder} there is a small box $C_I\times D_I$ for any small enough regular segment $I$. Then $\mu_\theta(C_I\times D_I)=|I\cap L(C_I\times D_I)|=\abs I$. As we shrink $I$, the measure of the small box goes to zero, so $(z,w)$ is not an atom.
\end{proof}
\section{Geometric properties of multi-foliation currents}\label{sec:propertiesmulti}
We study in this section geometric properties of the multi-foliation current. We discuss intersection patterns in the support of the current in Section \ref{s.kcrossing}, the intersection of the symmetrization of $\mu_\theta$ with closed curves in Section \ref{ss.intersection} and the link with measured lamination for translation structures induced by quadratic differentials in Section \ref{s.abelian}.
\subsection{Positive $k$-crossing}\label{s.kcrossing}
We discuss in this section further properties of the support of the multi-foliation current $\mu_\theta$ inspired by the work of Charlie Reid \cite[Theorem 1]{Reid}.

\begin{defn}
We say that a $k$-tuple $(x_1,y_1),\ldots,(x_k,y_k)\in G^+(\Sigma)$ form a \emph{positive k-crossing} if the $2k$-tuple $(x_1,\ldots, x_k,y_1,\ldots,y_k)\in (\partial\widetilde\Sigma)^{2k}$ is counterclockwise oriented. We say that $k$ oriented geodesics $l_1,\ldots, l_k\subset\wideS$ form a positive $k$-crossing if their endpoints form a positive $k$-crossing.
\end{defn}

\begin{figure}[htbp]
    \centering
    \includegraphics[width=0.3\linewidth]{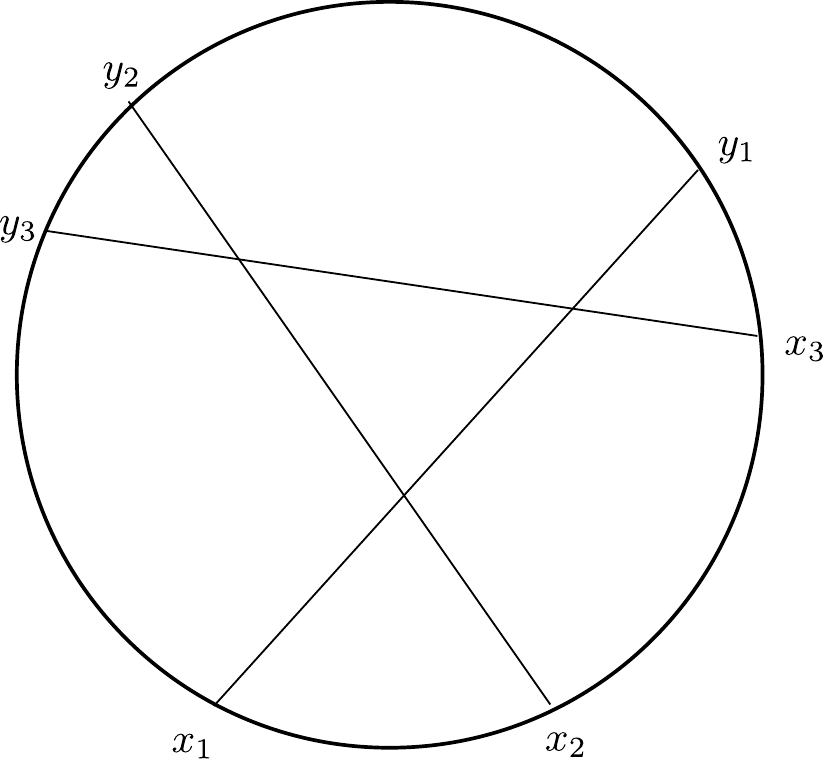}
    \caption{An example of 3-crossing.}
    \label{fig:placeholder}
\end{figure}

Note that if $k$ oriented geodesics form a positive $k$-crossing, then any ordered subset consisting of $m$ geodesics $(m\le k)$ forms a positive $m$-crossing. In particular, any pair of geodesics in a positive $k$-crossing intersect each other transversely.

\begin{defn}
We say that a geodesic current $\mu$ is \emph{$k$-crossing free} if its support, $\supp\mu$, contains no positive $k$-crossings.
%
    We say that a geodesic current $\mu$ is  \emph{maximal among $k$-crossing free geodesic currents} if it is $k$-crossing free, but for any  pair  $(x,y)\in G^+(\Sigma)\setminus \supp{\mu_\theta}$ linked to a pair in $\supp{\mu_\theta}$, $\supp{\mu_\theta}\cup\{(x,y)\}$ contains a positive $k$-crossing.
\end{defn}


\begin{thm}\label{thm:crossing_free}
Let $\Sigma$ be a $1/n$-translation surface. For every direction $\theta$, the associated current $\mu_\theta$ is maximal among $(\lfloor n/2\rfloor+1)$-crossing free geodesic current. 
\end{thm}

In order to prove Theorem \ref{thm:crossing_free}, it is enough to understand the possible positive intersection of straight leaves in a $1/n$-translation surface. For this we need a few geometric definitions.


\begin{defn}\label{d.angle}
 Let $l_1,l_2$ be leaves of the multi-foliation $\calF$ with endpoints $(x_1,y_1),(x_2,y_2)$ respectively. Assume that $l_1,l_2$ is a positive 2-crossing and $l_1\cap l_2=p$. The \emph{angle} $\angle (l_1,l_2)$ is
 \begin{enumerate}
\item the angle at $p$ formed by rays $py_1$ and $py_2$ if $l_1$ is left-turning;
\item the angle at $p$ between the rays $px_1$ and $px_2$ if $l_1$ is right-turning.
 \end{enumerate}
\end{defn}

\begin{figure}[htbp]
    \centering
    \includegraphics[width=0.65\linewidth]{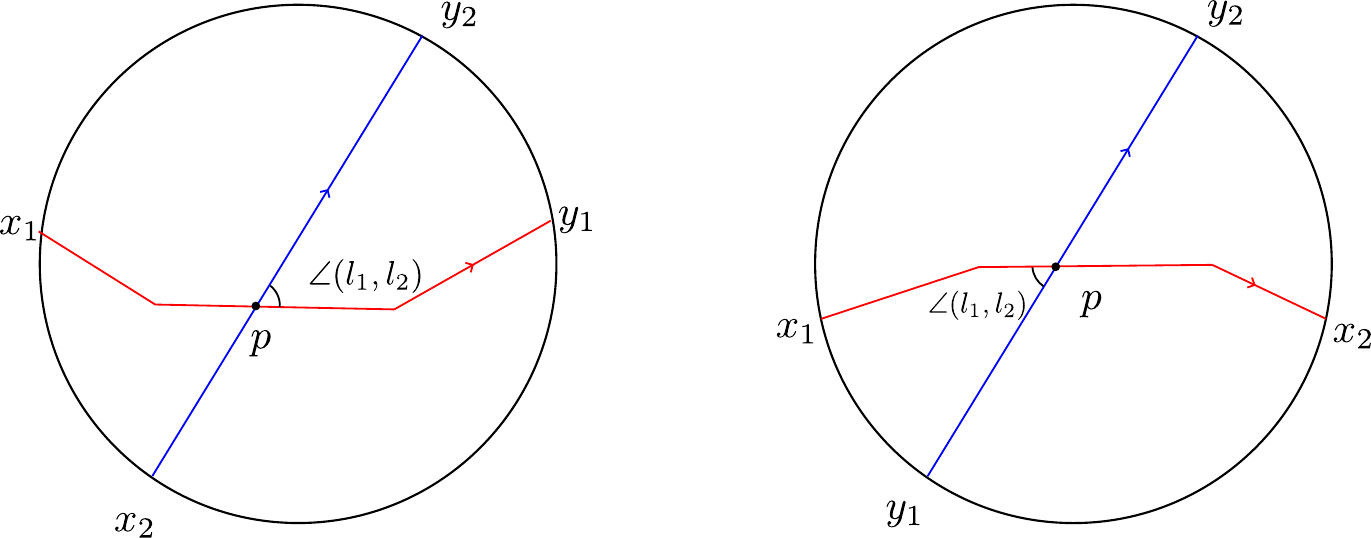}
    \caption{The angle of a positive 2-crossing.}
    \label{fig:placeholder}
\end{figure}

With these choices the angle is well defined regardless of $p$ being regular or not and it is always less than $\pi$. Furthermore the angle doesn't change reversing both orientations. 

\begin{lemma}\label{lem:cross_possible_angle}
If $l_1,l_2$ form a positive 2-crossing, then $\angle (l_1,l_2)=2k\pi/n$ for some $1\le k<n/2$.
\end{lemma}
\begin{proof}
    Assume $l_1$ is left-turning for simplicity. Then by Definition \ref{multi-foliation via local charts}, if two leaves $l_1,l_2$ intersect at a regular point $p$, then the angle between them is $2k\pi/n$ for some $1\leq k<n$. If they intersect at a cone point, since $l_1$ is left-turning, there is a Euclidean half disk neighborhood of $p$ on the left side of $l_1$, which contains sub-segments of $py_1$ and $py_2$, so the angle is still well defined and is a multiple of $2\pi/n$. 
    Since $l_1,l_2$ Form a positive 2-crossing, the angle $\angle (l_1,l_2)\in (0,\pi)$. Moreover, by Proposition \ref{p.transverse_intersection}, two straight leaves with transverse intersection cannot share saddle connections, so the angle cannot be 0 or $\pi$. Thus, we have $1\le k<n/2$.  
\end{proof}

\begin{prop}\label{lem:cross_angle_sum}
If $l_1,l_2,l_3$ are leaves of the multi-foliation $\calF$ and they form a positive 3-crossing, then $\angle (l_1,l_3)\ge\angle (l_1,l_2)+\angle (l_2,l_3)$.
\end{prop}
\begin{figure}[htbp]
    \centering
    \includegraphics[width=0.85\linewidth]{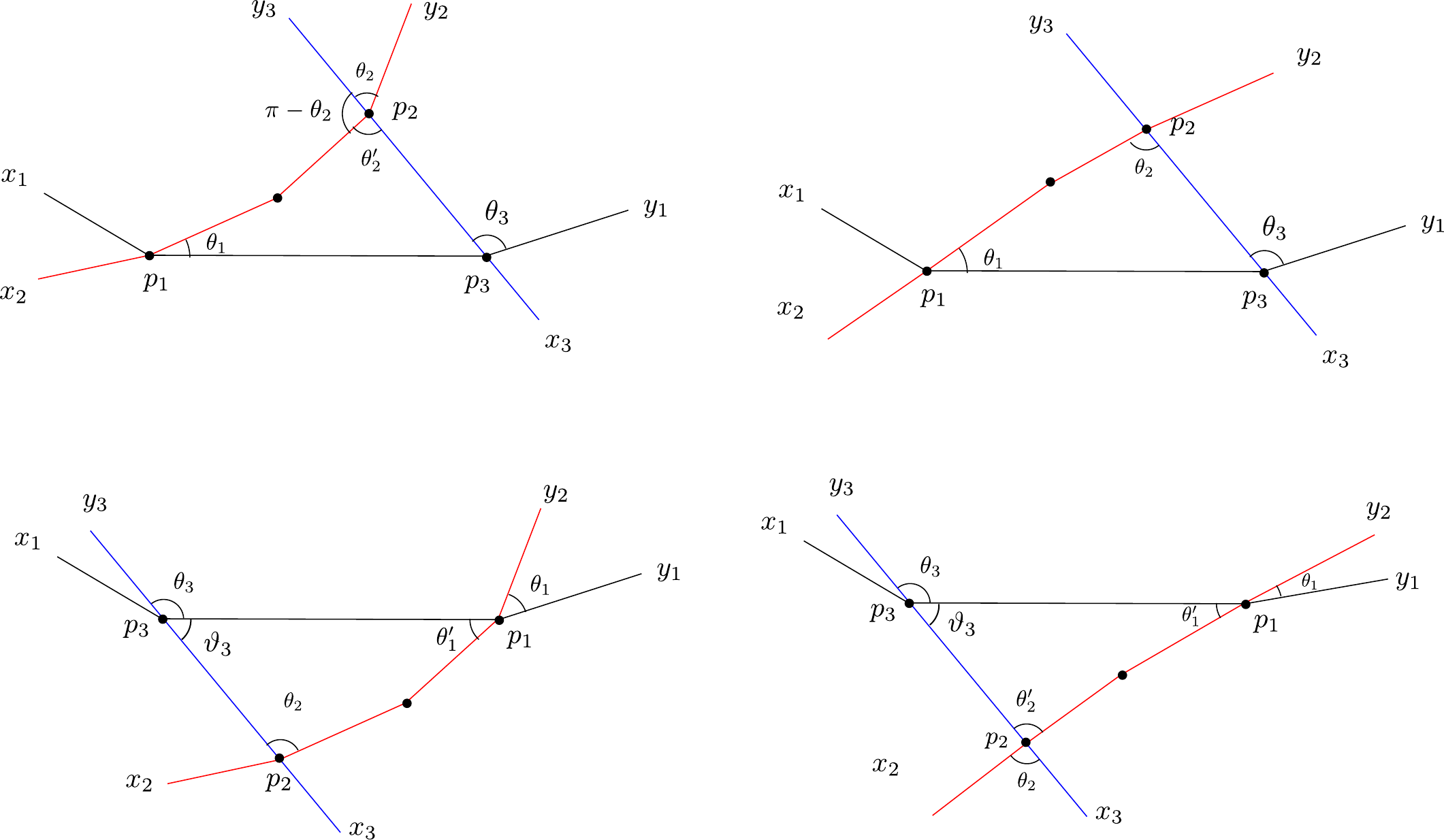}
    \caption{Different cases of a positive 3-crossing when $l_1$ is left turning. Top left: $l_2$ left-turning, $p_2$ on the left of $l_1$. Top right: $l_2$ right-turning, $p_2$ on the left of $l_1$. Bottom left: $l_2$ left-turning, $p_2$ on the right of $l_1$. Bottom right: $l_2$ right-turning, $p_2$ on the right of $l_1$.}
    \label{fig:3_crossing}
\end{figure}
\begin{proof}
    We may assume $l_1,l_2,l_3$ intersect pairwise at three different points, as intersection at the same point is the degenerate case and the argument follows from the general case.

    We further assume that $l_1$ is left-turning, the case in which $l_1$ is right-turning is symmetric. We distinguish four cases, depending on whether  $l_2$ is left or right-turning and $p_2=l_2\cap l_3$ is on the left or right side of $l_1$, see Figure \ref{fig:3_crossing}. Let $p_1=l_1\cap l_2$,  $p_3=l_1\cap l_3$ and $\theta_i$ be the angle of the two leaves intersecting at $p_i$, according to Definition \ref{d.angle}. We focus in each case on the geodesic triangle $T$ with vertices $p_1,p_2,p_3$. Note that albeit the sides of the triangle $T$ are geodesic,  they might contain cone points and make angles greater than $\pi$ there.

    Case 1) $l_2$ is left-turning and $p_2$ is on the left of $l_1$.  The angles at $p_1$ and $p_3$ are respectively $\theta_1$ and $\pi-\theta_3$. For the remaining angle $\theta_2'$, note that $l_2$ is left-turning, so the angle formed by rays $p_2x_2$ and $p_2y_3$ is $\pi-\theta_2$. Since $l_3$ is a geodesic, the sum $\theta_2'+(\pi-\theta_2)\ge\pi$ and thus $\theta_2'\ge \theta_2$. Apply Proposition \ref{CAT0 property} to $T$, we have $\theta_1+\theta_2'+(\pi-\theta_3)\le\pi$, which implies $\theta_3\ge \theta_1+\theta_2$  as desired.
    
    Case 2) $l_2$ is right-turning and $p_2$ is on the left of $l_1$.  In this case the angle of $T$ at $p_1$ is $\theta_1$, the angle of $T$ at $p_2$ is $\theta_2$ and the angle of $T$ at $p_3$ is $\pi-\theta_3$, so once again  Proposition \ref{CAT0 property}  implies that $\theta_1+\theta_2+(\pi-\theta_3)\le\pi$, or equivalently $\theta_3\ge \theta_1+\theta_2$ as desired.

    Case 3) $l_2$ is left-turning and $p_2$ is on the right of $l_1$.  In this case the angle  of $T$ at $p_2$ is $\theta_2$, while the angle $\theta_1'$ of $T$ at $p_1$ is opposite to $\theta_1=\angle (l_1,l_2)$ and thus satisfies $\theta_1'\geq \theta_1$: indeed $l_1$ is left turning and thus the angle between $p_1y_2$ and $p_1x_1$ is $\pi-\theta_1$, and $l_2$ is geodesic and so $\theta_1'+\pi-\theta_1\geq \pi$. Furthermore the angle $\vartheta_3$ of $T$ at $p_3$ is bigger or equal to $\pi-\theta_3$, so once again we have with Proposition \ref{CAT0 property}  $$\theta_1+\theta_2+(\pi-\theta_3)\le \theta_1'+\theta_2+\vartheta_3 \le\pi,$$ which implies $\theta_3\ge \theta_1+\theta_2$ as desired.

    Case 4) $l_2$ is right-turning and $p_2$ is on the right of $l_1$. In this case the angle $\theta_1'$ of $T$ at $p_1$ is opposite to $\theta_1$ and thus we have $\theta_1'\geq \theta_1$ and for the same reason the angle $\theta_2'$ of $T$ at $p_2$ satisfies $\theta_2'\geq \theta_2$. Furthermore the angle $\vartheta_3$ of $T$ at $p_3$ is bigger or equal to $\pi-\theta_3$, so once again we have with Proposition \ref{CAT0 property}  $$\theta_1+\theta_2+(\pi-\theta_3)\le \theta_1'+\theta_2+\vartheta_3 \le\pi,$$ which implies $\theta_3\ge \theta_1+\theta_2$ as desired.        \end{proof}

\begin{proof}[Proof of Theorem \ref{thm:crossing_free}]
    Let $m=\lfloor n/2\rfloor+1$. Suppose there is a positive $(m+1)$-crossing formed by leaves $l_1,\ldots,l_{m+1}$. By Lemma \ref{lem:cross_possible_angle}, the angle of any positive 2-crossing is a non-zero multiple of $2\pi/n$. By induction, using Proposition \ref{lem:cross_angle_sum}, the angle $\angle (l_1,l_{m+1})$ is at least $m\cdot 2\pi/n>\pi$, which is a contradiction. Thus, the geodesic current $\mu_\theta$ is $(\lfloor n/2\rfloor+1)$-crossing free.

 In order to discuss maximality among $(\lfloor n/2\rfloor+1)$-crossing free currents we first show that for every pair $(x,y)$ in $\supp{\mu_\theta}$, for every leaf $l$ with endpoints $(x,y)$ and every  point $p$ in $l$ there is a positive $\lfloor n/2\rfloor$-crossing contained in $\supp{\mu_\theta}$, such that any pair in the crossing is an endpoint of a leaf passing through $p$. 
 Indeed the $\lfloor n/2\rfloor$ left-turing straight leaves through $p$ consecutive to $l$ form a   positive $\lfloor n/2\rfloor$-crossing,  and they belong to the support of $\mu_\theta$ by Proposition \ref{p.support}. 

 Turning to maximality let $(x_0, y_0)\in G^+(\Sigma)\setminus \supp\mu$ which is linked with a leaf $(x,y)\in\supp{\mu_\theta}$. Choose leaves $l_0$ and $l$ with endpoints $(x_0,y_0)$ and $(x,y)$ respectively. Since $l$ is a leave $l\cap l_0=p$ and the two geodesics intersect transversely at $p$ (recall Proposition \ref{p.transverse_intersection}). We consider all  the left-turning straight leaves $l_i$ of $\calF$, including $l$, that contain $p$ and from  an angle less then $\pi/2$ with  $\overline l$ at $\overline p$. There are precisely $\lfloor n/2\rfloor$ of them and, together with $\overline l$ they make a positive $(\lfloor n/2\rfloor+1)$-crossing.
 
\end{proof}

\subsection{Intersection number of measured multi-foliation}\label{ss.intersection}
There is a natural involution on $G^+(\wideS)$ given by $(x,y)\mapsto (y,x)$. We consider  the \emph{space of unoriented geodesics}\[
G(\widetilde\Sigma):=(\partial\wideS \times \partial\wideS)\setminus\Delta/_{(x,y)\sim (y,x)}.\]

Let $\pi:G^+(\wideS)\to G(\wideS)$ be the  map  forgetting the ordering. 
\begin{defn}
    A \emph{geodesic current} is a $\pi_1(\Sigma)$-invariant, locally finite, Borel measure on $G(\wideS)$.\footnote{Some reference requires a geodesic current to be regular. By \cite[II, Theorem 3.3]{malliavin2012integration}, any locally finite $\sigma$-compact Borel measure defined on $G^+(\wideS)$ or $G(\wideS)$ will be automatically regular.} We denote by  $\mc C(\Sigma)$ the space of (symmetric) geodesic currents.
\end{defn}
We denote by $\mu^s=\pi_*\mu$ the  geodesic current associated to an oriented geodesic current $\mu$, in other words $\mu^s(A)=\mu(\pi^{-1}(A))$ for any Borel subset $A\subset G(\wideS)$.
The  symmetric geodesic current $\mu_\theta^s$ induced by  the multi-foliation current $\mu_\theta$ has support $\pi(\supp{\mu_\theta})=\pi(G^+(\til{\mc F}))$, the set $G(\til{\mc F})=\pi(G^+(\widetilde\calF))$ of unordered pairs of endpoints of leaves in $\wideF$.

Recall from Definition \ref{d.smallinterval} that a regular segment $I\subset\wideS$ is a small interval if $C_I\times D_I$ is a small box.
\begin{lemma}\label{l.mus}
Let $I$ be small. Then
$$\mu_\theta^s(\pi(C_I\times D_I))=\left\{\begin{array}{ll}|I|&\text{ if $n$ is odd}\\
2|I|&\text{ if $n$ is even}
\end{array}\right.$$
\end{lemma}
\begin{proof}
By definition of $\mu_\theta^s$, it holds $\mu_\theta^s(\pi(C_I\times D_I))=\mu_\theta(C_I\times D_I)+\mu_\theta(D_I\times C_I)$. When $n$ is even $\mu_\theta(D_I\times C_I)=\mu_\theta(C_I\times D_I)$ because every leaf comes with two opposite orientations, so that $\mu_\theta^s(\pi(C_I\times D_I))=2|I|$, while when $n$ is odd $\mu_\theta^s(C_I\times D_I)=\mu_\theta(C_I\times D_I)=|I|$. 
\end{proof}


We denote by $\delta_\gamma$ the current  associated to a closed curve: Let $\gamma\subset \Sigma$ be a closed curve, the preimage in $\wideS$ under the covering map $\wideS\to \Sigma$ of the geodesic representative of $\gamma$ is a countable disjoint union of bi-infinite geodesics. The sum of the Dirac masses on their endpoints, is a well defined locally finite Borel measure since the endpoints form a discrete subset of $G(\wideS)$,  and is clearly $\pi_1(\Sigma)$-invariant. 

Bonahon constructed in \cite{bonahon1988geometry} a bilinear function on $\mc C(\Sigma) \times \mc C(\Sigma)$ which extends the geometric intersection number. Recall that two pairs $(x_1,y_1),(x_2,y_2)\in G(\widetilde\Sigma)$ are \emph{linked} if $x_2,y_2$ are in different connected components of $\partial\wideS\setminus\{x_1,y_1\}$. Let $\mc P\subset G(\wideS)\times G(\wideS)$ be the subset containing linked pair of endpoints. Two geodesic currents $\alpha,\beta$ induce a $\pi_1(\Sigma)$-invariant product measure $\alpha\times \beta$ on $\mc P$, which descends to a measure on $\mc P/\pi_1(\Sigma)$ still denoted by $\alpha\times \beta$.

\begin{thm}[\cite{bonahon1988geometry}]\label{intersection.number}
    The function 
    \[\begin{array}{cccl}
    i: &\mc C(\Sigma) \times \mc C(\Sigma) &\to &\R\\
    &(\alpha,\beta)&\mapsto&\int_{\mc P/\pi_1(\Sigma)}\alpha\times\beta
    \end{array}
    \]
    is bilinear and extends the geometric intersection number on currents associated to closed curves.
\end{thm}


We now give a concrete way to compute the intersection of $\mu_\theta$ with a closed CAT(0)-geodesic $\gamma\subset\Sigma$.  We choose a lift $\til \gamma\subset\til\Sigma$ of $\gamma$ and a cone point $x_0\in\til\gamma$ (we can assume that it exists by Proposition \ref{p.closed_conepoint}). The geodesic segment $J_\gamma=[x_0,\gamma.x_0)$ is a fundamental domain of the $\gamma$-action on $\til\gamma$. 
To simplify the notation, we choose an orientation of $\til\gamma$, and orient the geodesics transverse to it so that they go from the left to the right. 

\begin{defn}\label{d.transverse_intersection_segment}
    We say a geodesic $l$ is \emph{transverse} to $\til\gamma$ at $J_\gamma$ if $l$ is transverse to $\til\gamma$ and one of the following holds:
    \begin{itemize}
        \item if $l$ belongs to a cylinder, then the central leaf of the cylinder crosses $J_\gamma$;
        \item otherwise the last point in $l\cap\til\gamma$ belongs to $J_\gamma$.
    \end{itemize}
\end{defn}
We denote by $G(J_\gamma)\subset G(\wideS)$ the unordered pairs of endpoints of geodesics transverse to $\til\gamma$ at $J_\gamma$, and $G^\mc F(J_\gamma)= G(J_\gamma)\cap G(\wt {\mc F})$.

\begin{prop}\label{p.fundamental_domain}
    It holds 
    \[i(\mu_\theta^s,\delta_\gamma)=
    \mu_\theta^s(G^\mc F(J_\gamma)).\]
\end{prop}

\begin{proof}
Since $\delta_\gamma$ is supported on  endpoints of lifts of $\gamma$, it suffices to compute the $\mu_\theta^s\times \delta_\gamma$-measure of a fundamental domain of 
    \[\mc P_\gamma=\{(\underline{x},\underline{y})\in\mc P |\;\underline{x},\underline{y}\in G(\wideS), \;\underline{y} \text{ are the endpoints of a lift of }\gamma  \}\]
 We claim that $\mc P(J_\gamma):=G(J_\gamma)\times\{ (\til\gamma^+,\til\gamma^-)\}\subset \mc P$
   is a strict fundamental domain of $\mc P_\gamma/\pi_1(\Sigma)$, and thus \[i(\mu_\theta^s,\delta_\gamma)=
    \int_{\mc P(J_\gamma)}\mu_\theta^s\times\delta_\gamma=
    \mu_\theta^s(G(J_\gamma))=\mu_\theta^s(G^\mc F(J_\gamma)),\]
where the last equality holds because $\supp{\mu_\theta^s}= G(\til\calF)$ (Proposition \ref{p.support}).

    If there existed a non-trivial element $\eta\in\pi_1(\Sigma)$ such that $\eta.\mc P(J_\gamma)\cap \mc P(J_\gamma)\ne\emptyset$, then $\eta.(\til\gamma^+,\til\gamma^-)=(\til\gamma^+,\til\gamma^-)$ implies that $\eta$ is a non-trivial power of $\gamma$, but since $J_\gamma$ is a strict fundamental domain for the $\gamma$-action on $\til\gamma$, $\gamma^kG(J_\gamma)\cap G(J_\gamma)=\emptyset$ unless $k=0$. 
    

    In order to conclude the proof, we need to show that $\mc P_\gamma=\bigcup \eta.\mc P(J_\gamma)$. Given any $(\underline{x},\underline{y})\in \mc P_\gamma$, $\underline{y}$ consists of endpoints of a lift $\til\gamma$ of $\gamma$. There exists an $\eta\in\pi_1(\Sigma)$ such that $\eta.\underline{y}$ coincides with the endpoints of $\til\gamma$. Then composing $\eta$ with a suitable power of the deck transformation corresponding to $\gamma$ (which is again denoted by $\eta$ by abusing notation), we may assume that the geodesic corresponding to $\eta.\underline{x}$ intersects $\til\gamma$ at a point in $J_\gamma$. Thus, $(\eta.\underline{x},\eta.\underline{y})\in \mc P(J_\gamma)$, which implies that $(\underline{x},\underline{y})\in \eta^{-1}.\mc P(J_\gamma)$.
\end{proof}



We want to cover $\pi^{-1}(G^\mc F(J_\gamma))$ with small boxes using Proposition \ref{l.smallbox_cylinder} and \ref{l.smallbox_noncylinder}. Let $x_0,x_1,\ldots,x_s$ be the ordered  cone points in $J_\gamma$ by  where $x_s=\gamma.x_0$, and denote by $J_{j}=(x_{j-1},x_j)$ the open saddle connection. Let $L_j\subset G^+(\wideS)$ be the set of leaves transverse to $\til\gamma$ at $J_j$. The leaves in $L_j$ can intersect $J_j$ in $n$ possible angles $\theta_j+\frac{2m\pi}{n}, m=1,\ldots,n$, for some fixed $\theta_j$ depending on the slope of $J_j$. Denote by $L_{j,k}$ the set of leaves transverse to $J_j$ with intersection angle $\theta_j+\frac{2k\pi}{n}$ and by $G_{j,k}$ the pairs of endpoints of leaves in $L_{j,k}$.

\begin{lemma}\label{l.set_disjoint}
    If $(j,k)\ne (l,m)$, then $G_{j,k}\cap G_{l,m}=\emptyset$.
\end{lemma}

\begin{proof}
    If $k\ne m$, then $G_{j,k}\cap G_{j,m}=\emptyset$ because two leaves with the same endpoints are parallel (Proposition \ref{p.bound_cylinder}), in particular they intersect $J_j$ with the same angle.

    Assume $j\ne l$.  Let $L(\wideF)$ be the set of bi-infinite straight leaves of $\wideF$ and consider the endpoint map $\partial:L(\wideF)\to G(\wideS)$. Of course  $L_{j,k}\subset \partial\inv(G_{j,k})$.
    We will prove that $\partial\inv(G_{j,k})\cap \partial\inv(G_{l,m})=\emptyset$. Of course $L_{j,k}\cap L_{l,m}=\emptyset$, furthermore a leaf  $l\in \partial\inv(G_{j,k})\setminus L_{j,k}$ passes through either $x_{j-1}$ or $x_j$ and has the same endpoints as a leaf in $L_{j,k}$.  By Proposition \ref{p.bound_cylinder}, $l$ is thus the  boundary of a cylinder. If, by contradiction, also $\partial l\in G_{l,m}$, then $l=j\pm1$, we assume without loss of generality $l=j+1$, or equivalently $x_j\in l$.  Moreover, $l$ is the boundary leaf of a cylinder intersecting $J_j$ and another cylinder intersecting $J_{j+1}$. Then it must be both left-turning and right-turning. But this contradicts the assumption that $x_j$ is a cone point. 
\end{proof}
\begin{lemma}\label{l.measure_decomposition}
It holds\[
    \mu_\theta^s(G^\mc F(J_\gamma))=\sum_{j,k}\mu_\theta(G_{j,k}).
    \]
\end{lemma}

\begin{proof}
    By disjointness (Lemma \ref{l.set_disjoint}) it holds $\mu_\theta(\cup G_{j,k})=\sum_{j,k}\mu_\theta(G_{j,k})$. The set $\pi^{-1}(G^\mc F(J_\gamma))\setminus \big(\cup G_{j,k}\big)$ consists of endpoints of leaves that contain one of the cone points and do not belong to any cylinder, otherwise a parallel leaf with the same endpoints would belong to $L_{j,k}$.  By Lemma \ref{l.atomic}, none of such point is atomic. Since there are only finitely many leaves passing through each cone point,  the set $\pi^{-1}(G^\mc F(J_\gamma))\setminus \big(\cup G_{j,k}\big)$ is a null set with respect to $\mu_\theta$. Furthermore, by definition of $\mu_\theta^s$ it holds $\mu_\theta^s(G^\mc F(J_\gamma))=\mu_\theta(\pi^{-1}(G^\mc F(J_\gamma)))$.
\end{proof}
The following will be key in the computation of the measure of $G_{j,k}$.
Recall from Proposition \ref{l.smallbox_noncylinder} that a boxing rectangle is a Euclidean rectangle, whose horizontal sides are contained in leaves of $\calF$ that are not regular cylinder leaves, and such that the diagonal of the rectangle makes an angle not bigger than $2\pi k/n$ with the horizontal sides. In this case by Proposition \ref{l.smallbox_noncylinder} the vertical sides $I$ of the rectangle are small intervals and $C_I\times D_I$ is a small box with $I$ as measuring interval.
\begin{lemma}\label{l.subdivision}
Given a saddle connection $J$ and a direction $v_k$ of $v_{[\theta]}$ transverse to $J$ there is a subdivision $J=\bigcup_\alpha J_\alpha$ such that for any $\alpha$ one of the following holds
\begin{enumerate}
    \item $J_\alpha$ is precisely transverse to a cylinder in direction $v_k$, namely the endpoints of $J_\alpha$ lie in the boundary leaves of the cylinder
    \item $J_\alpha$ is contained in a boxing rectangle. 
\end{enumerate}
\end{lemma}
\begin{proof}
Since there are only finitely many cylinders parallel to $\calF$ in $\Sigma$, there is a lower bound on the width of each such cylinder, which implies that $J$ only intersects finitely many cylinders. We consider a first finite refinement of $J=\bigcup J_\beta$ where $J_\beta$ is either the intersection of $J$ with a cylinder in direction $v_k$, or is a complementary subinterval. Let now $J_\beta$ be a subinterval of the second kind. Then $J_\beta$ does not intersect any regular cylinder leaf. Since the set of cone points is discrete there is a lower bound on length of segments in direction $v_k$ with one endpoint in $J$ and the other in a saddle connection. This implies that, up to refining the decomposition, we can assume that each subinterval is contained in a boxing rectangle.
\end{proof}

\begin{lemma}\label{l.projection_length}
    The measure of $G_{j,k}$ is given by \[
    \mu_\theta(G_{j,k})=\abs{J_j}\left|{\sin\left(\theta_j+\frac{2k\pi}{n}\right)}\right|.
    \]  
\end{lemma}
\begin{proof}
Consider the subdivision $J_j=\bigcup_\alpha J_\alpha$ given by Lemma \ref{l.subdivision}. We denote, for each $\alpha$, by $I_\alpha$ an interval orthogonal to $v_k$ with endpoints in the same leaf as $J_\alpha$, and by $B_\alpha\subset G^+(\wideS)$ the small box $B_\alpha=C_{I_\alpha}\times D_{I_\alpha}$. If $J_\alpha$ corresponds to a cylinder then $B_\alpha=\{(l^-,l^+)\}$.  By construction for $\alpha\neq\beta$ $B_\alpha\cap B_\beta$ is either empty or consists of the endpoints of the non-atomic regular leaf through the intersection $J_\alpha\cap J_\beta$ so $\mu_\theta(B_\alpha\cap B_\beta)=0$. Furhtermore $G_{j,k}=\bigcup_\alpha B_\alpha$, and $|I_\alpha|=|J_\alpha|\abs*{\sin\paren*{\theta_j+\frac{2k\pi}{n}}}$.  We then have
\begin{align*}
        \mu_\theta(G_{j,k})&=\mu_\theta(\cup_\alpha  B_\alpha)=\sum_\alpha\mu_\theta( B_\alpha)=\sum_\alpha\abs{I_\alpha}=\sum_\alpha\abs{J_\alpha}\left|{\sin\left(\theta_j+\frac{2k\pi}{n}\right)}\right|\\
        &=\abs{J_j}\left|{\sin\left(\theta_j+\frac{2k\pi}{n}\right)}\right|.
    \end{align*}
\end{proof}

We now have all the ingredients to compute the intersection of the geodesic current $\mu_\theta^s$ with a closed curve.  Recall from Definition \ref{d.follength} that, for a closed piecewise $C^1$-curve $\gamma\subset\Sigma$, we denote by $\ell_\theta([\gamma])$ the infimum of the length function of closed curves freely homotopic to $\gamma$ with respect to the metric given by the measured multi-foliation, this agrees with the Finsler norm induced by the $(\theta+\frac\pi2)$-web. We can always take a CAT(0) geodesic representative of $\gamma$ to compute the length by Theorem \ref{thmA}.

\begin{thm}\label{foliation intersection}
The multi-foliation current $\mu_\theta^s$ satisfies 
\[
i(\mu_\theta^s,\delta_\gamma)=\ell_{\theta}([\gamma]).
\]
\end{thm}

\begin{proof}
    By Proposition \ref{p.fundamental_domain}, Lemma \ref{l.measure_decomposition} and \ref{l.projection_length}, we have \[
    i(\mu_\theta^s,\delta_\gamma)=\mu_\theta^s(G^\mc F(J_\gamma))=\sum_{j,k}\mu_\theta(G_{j,k})=\sum_{j,k}\abs{J_j}\left|{\sin\left(\theta_j+\frac{2k\pi}{n}\right)}\right|.\]
    By Corollary \ref{theta length}, the length function $\ell_\theta$ is induce by the inner product with the $\theta+\frac{\pi}{2}$-web (in fact $\frac{\pi}{2}$-web since we assume $\mc F$ is horizontal). Therefore, \[
    \ell_\theta(I_j)=\sum_k\abs{I_j}\left|\cos\left(-\theta_j+\frac{\pi}{2}+\frac{2k\pi}{n}\right)\right|=\sum_{k}\abs{J_j}\left|{\sin\left(\theta_j+\frac{2k\pi}{n}\right)}\right|.\]
    Which shows the result since  $J_\gamma$ is the concatenation of the saddle connections $J_j$.
%
\end{proof}

\subsection{Relation with measured laminations for special $1/n$-translation structures}\label{s.abelian}
We will show the $1/n$-translation structure is induced by an (half)-translation structure, the $\theta$-multi-foliation current $\mu_\theta^s$ can be expressed as a sum of measured laminations.

\begin{defn}\label{d.measured lamination} A \emph{measured lamination} $\lambda$ is a geodesic current with $i(\lambda,\lambda)=0$.  An \emph{oriented measured lamination} is oriented geodesic current $\lambda_0$ whose image $\lambda^s_0$ under the symmetrization operator  satisfies $i(\lambda^s,\lambda^s)=0$.
\end{defn}
See \cite[Proposition 8.3.9]{martelli2016introduction} for the relation with more classical definition of measured laminations on surfaces. In the case of (half)-translation structures, the multi-foliation $\calF_\theta$ is a singular measured foliation on the surface and our construction of the (multi)-foliation current specializes to the classical correspondence between measured foliations and measured laminations (compare \cite[\S 11.3]{kapovich2001hyperbolic}). 

Note that it follows from Theorem \ref{thm:crossing_free} that if $n=1,2$ then the oriented (multi)-foliation current is an oriented measured lamination:
\begin{prop}\label{p.lamination}
    An oriented geodesic current $\mu$ is $2$-crossing free if and only if it is an oriented measured lamination.
\end{prop}
\begin{proof}
    If $\mu$ is an oriented measured lamination, then its support contains no intersecting geodesics, so it is $2$-crossing free. Conversely, if $\mu$ is  $2$-crossing free, then pairs of geodesics in its support do not intersect transversely, so $i(\mu^s,\mu^s)=0$ and $\mu$ is an oriented measured lamination.
\end{proof}

We say that a $1/n$-translation structure is \emph{induced by a $1/m$-translation structure} (here $n=km$ for some natural number $k$) if its regular locus $\Sigma\setminus P$ admits an $(\R^2\rtimes\Z/ m\Z,\R^2)$-structure compatible with the $1/n$-translation structure, or equivalently there exists a subatlas of the atlas defining the $1/n$-translation structure whose change of charts are in $\R^2\rtimes\Z/ m\Z$. Of course for every $m,k$ an $1/m$-translation structure induces a $1/mk$-translation structure, while the $1/km$-translation structures induced by a $1/m$-translation structure form a smaller locus in the relevant moduli space. 

Note that, if the $1/n$ translation structure is induced by a $1/m$-translation structure, then for any angle $\theta$ the multifoliation $\calF_\theta$ is the union of $k$ globally well defined $m$-multifoliations.
We denote by $\mu_\theta^n$ (respectively $\mu_\theta^m$) the $\theta$-multifoliation induced by the $1/n$- (respectively $1/m$-)translation structure.

\begin{prop}\label{degenerate multi-foliationGen}
If $\Sigma$ is a $1/mk$-translation surface induced by a $1/m$-translation structure, then
\[\mu_\theta^n=\sum_{j=1}^{k}\mu^m_{\theta+\frac{2\pi k}n}.\]
\end{prop}

\begin{proof}
This follows readily from our construction. Indeed if $n=mk$ the $(n,\theta)$-web $v_{[\theta]}^n$ can be written as the  union of $k$ disjoint $(m,\theta)$-webs
$$v_{[\theta]}^n=\bigsqcup_{j=1}^k v_{[\theta+\frac{2\pi j}n]}^m.$$

This decomposition induces a decomposition of the leaves of the multifoliation $\til\calF_\theta^n$ as a disjoint union of leaves of the foliations $\til\calF_{\theta+\frac{2\pi j}n}^m$, which are globally well defined because the $1/n$-translation structure is induced by a $1/m$-translation structure. 

This decomposes the support of $\mu_\theta^n$ as a disjoint union of the supports $S_j$ of the currents $\mu_{\theta+\frac{2\pi j}n}^m$ and it is a direct consequence of the construction that $\mu_\theta^n|_{S_j}=\mu_{\theta+\frac{2\pi j}n}^m$.
 %
\end{proof}
Consistently with the literature, we say that a $1/n$-translation structure is induced by a (half)-translation structure if it is induced by a $1$ (resp. $1/2$-)translation structure.

\begin{cor}\label{degenerate multi-foliation}
If the $1/n$-translation structure is induced by a (half)-translation structure on $\Sigma$, then the multifoliation current $\mu^s_\theta$ is a sum of $n$  measured laminations.
\end{cor}

\section{Liouville currents for Finsler metrics}\label{sec:Liouville_current}

Fix a metric $g$ on a surface $\Sigma$, and denote by $\ell_g$ the associated length function, so that, for every closed curve $\gamma$, $\ell_g(\gamma)$ is is the smallest length with respect to the metric $g$ of a closed curve in the homotopy class of $\gamma$. We say that a geodesic current $\mu\in\mc C(\Sigma)$ is  the \textit{Liouville current} of $g$ if $i(\mu,\delta_\gamma)=\ell_g(\gamma)$ for all $\gamma\in\pi_1(\Sigma)$. 
Note that by \cite{otal1990spectre} a geodesic current is uniquely determined by its intersection number with all closed geodesics, so if a Liouville current exists, it is unique.

The following example will be key:
\begin{ex}
Let, as in Section \ref{sec:Finsler}, $F_\theta$ be the $\theta$-metric, namely the Finsler metric associated with a $(\theta+\frac{\pi}{2})$-web $v_{[\theta+\frac{\pi}{2}]}$ on a $1/n$-translation surface $\Sigma$. By Theorem \ref{foliation intersection}, the multi-foliation current $\mu_\theta^s$ is the Liouville current for $F_\theta$.
\end{ex}

In this section, we  prove Theorem \ref{current_quadratic_differential}: any compatible Finsler metric on a $1/n$-translation surface admits a Liouville current, and more precisely we show that the Liouville currents can be expressed as an integral combination of the symmetrized version of $\theta$-multi-foliations.
The result will follow from a detour in convex geometry in $\mathbb R^2$.



\subsection{Polygonal norms on $\mathbb R^2$}
In this subsection we show that all symmetric polygonal norms on $\R^2$ arise as finite convex combinations of $n$-webs (Corollary \ref{discrete measure}). We will generalize this result to all norms in the next subsection (see Proposition \ref{measure_for_vector}).

\begin{prop}\label{finite vectors}
    For any symmetric norm $\norm{\cdot}$ on $\R^2$ whose unit sphere is a polygon,  there exists finitely many vectors $v_i$ such that for every $v\in\R^2$ it holds $\norm v=\sum_i \abs{\innerproduct{v}{v_i}}$.
\end{prop}

\begin{figure}[h]
    \centering
    \includegraphics[width=0.3\linewidth]{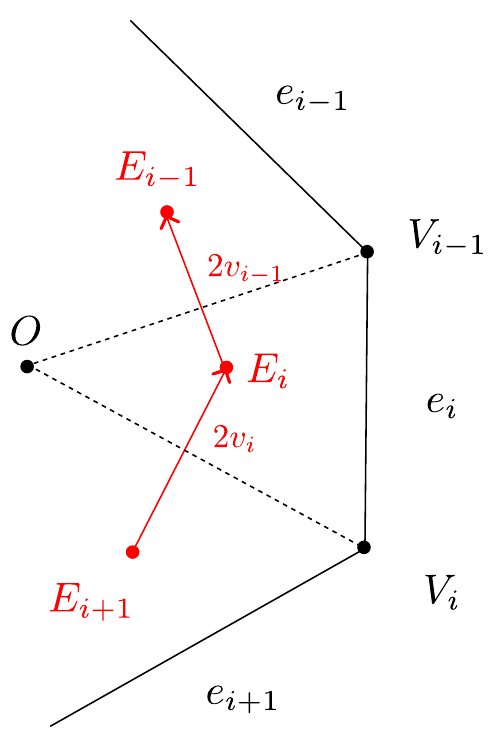}
    \caption{The polygon $P$ in black and its dual $P^*$ in red.}
    \label{polygonanddual}
\end{figure}

\begin{proof}
    Observe that, as soon as there are two linearly independent vectors in the collection $v_i$, the function $\sum_i \abs{\innerproduct{v}{v_i}}$ is a norm, indeed it is positive, positively homogeneous and satisfies the triangular inequality. We will refer to it as a type one norm.   Since a norm is uniquely determined by its unit sphere, it is enough to construct, for every polygon $P$ a type one norm with $P$ as unit sphere. 
    
    We fix some notation as illustrated in Figure \ref{polygonanddual}. We will use the standard coordinates $(x,y)$ on $\R^2$. We label the edges of $P$ in clockwise order as $e_1,e_2,\ldots,e_{2n}$, and let $a_i,b_i\in \R$  be so that  $e_i$ has the equation $a_ix+b_iy=1$. The pair $E_i=(a_i,b_i)$ is a vertex of the dual polygon $P^*$ and the line spanned by $(a_i,b_i)$ is orthogonal to the edge $e_i$. Considering the indices modulo $2n$, we denote the common vertex of $e_{i-1}$ and $e_i$ by $V_{i-1}$. 


    Since the polygon $P$ is symmetric, the dual polygon $P^*$ is symmetric as well, that is $(a_i,b_i)=-(a_{i+n},b_{i+n})$. For every $1\leq i\leq 2n$ we set $v_i=\frac{1}{2}(a_{i}-a_{i+1},b_i-b_{i+1})$. The vector $v_i$ is the half-edge of $P^*$ and is orthogonal to the ray spanned by the vertex $V_i$. Observe that by construction $v_i=-v_{i+n}$.

    We consider the type one norm\[
    \norm v= \sum_1^{n}\abs{\innerproduct{v}{v_i}}
    \]
    where we only sum from $1$ to $n$. Note that the rays $OV_i$ through the vertices of $P$ divide the plane in sectors, and since the vectors $v_i$ are orthogonal to the respective rays the function $\norm v$ is linear on each sector. Let us then focus on the $i$-th sector $S_i$ bound by the rays $OV_{i-1}$ and $OV_i$. For a vector $v=(x,y)$ in  this sector it holds 
\[
    \begin{cases}
        \innerproduct{v}{v_k}>0,& k=i,i+1,\ldots,n\\
        \innerproduct{v}{v_k}<0,& k=1,2,\ldots, i-1\\
    \end{cases}
    \]
    Thus, $\norm{v}=1$ descends to 
    \begin{align*}
        1 &=\norm{v}=\sum_{j=i}^n\innerproduct{v}{v_j}-\sum_{k=1}^{i-1}\innerproduct{v}{v_k} \\
        &=\frac{1}{2}(a_i-a_{n+1})x+\frac{1}{2}(b_i-b_{n+1})y-\frac{1}{2}(a_{1}-a_{i})x-\frac{1}{2}(b_{1}-b_i)y \\
        &=a_{i}x+b_{i}y 
    \end{align*}
where in the  last equality we used that $(a_1,b_1)=-(a_{n+1},b_{n+1})$. This shows that  the set of vectors of norm one in the sector $S_i$  are precisely the vectors in the edge $e_i$, which concludes the proof.
\end{proof}

From the proof above, one can also get the following corollary.

\begin{cor}\label{c.polygonalnorms}
    For every symmetric norm $\norm{\cdot}$ on $\R^2$ whose unit sphere is a polygon $P$, there exists a finite sum $\nu$ of Dirac masses  such that if  $v_\theta$ is the vector in direction $\theta$ with unit Euclidean norm, then \[
    \norm{v}=\int_{\mb S^1} \abs{\innerproduct{v}{v_{\theta+\frac{\pi}{2}}}} d\nu.
    \]
    More precisely, $\nu=\sum  K_i\delta_{V_i}/4$, where $V_i$ is the vertex of $P$, $\delta_{V_i}$  is the Dirac measure supported at the direction $OV_i$, and and $K_i$ is  the length of the edge of the dual polygon $P^*$ corresponding to the vertex $V_i$ of $P$.
\end{cor}
In other words if the edge $e_i$ between $V_{i-1}$ and $V_i$ satisfies $a_ix+b_iy=1$, then $$K_i=\sqrt{(a_i-a_{i-1})^2+(b_i-b_{i-1})^2}.$$
\begin{proof}
    From the proof of Proposition \ref{finite vectors}, we know $\norm{v}=\sum_1^n\abs{\innerproduct{v}{v_i}}=\frac{1}{2}\sum_1^{2n}\abs{\innerproduct{v}{v_i}}$. Moreover the vector $v_i$ is orthogonal to the ray $OV_i$.
\end{proof}

If the Finsler metric has additional rotation symmetry, we can similarly represent it via inner product with an $n$-web in the sense of Definition \ref{web_definition}.

\begin{cor}\label{discrete measure}
For every symmetric norm $\norm{\cdot}$ on $\R^2$ invariant under $\Z/ n\Z$ and whose unit sphere is a polygon $P$, there exists a finite sum $\nu$ of Dirac masses such that if  $v_{[\theta]}$ is the $n$-web in direction $\theta$ 
\[
    \norm{v}=\int_{\mb S^1} \abs{\innerproduct{v}{v_{[\theta+\frac{\pi}{2}]}}} d\nu.
    \]
    With the notations of Corollary \ref{c.polygonalnorms}, $\nu=\sum  K_i\delta_{V_i}/4n$. 
\end{cor}
\subsection{General norms on $\R^2$}
In order to treat  general  symmetric norms on $\R^2$ that are invariant under $\Z/n \Z$ we consider successive  approximations via polygons. 

We fix some notation. 
Choose a countable dense subset $\{p_i\}$ of the unit sphere of $\norm\cdot$ which is invariant under $\pi$ and $\frac{2\pi}{n}$ rotation. We say that two points are in the same equivalence class if they are related by a rotation of angle $a\pi+b\frac{2\pi}{n}$ for some $a,b\in\Z$. We get in this way countably many equivalence classes of points, denoted by $[p_1],[p_2],\ldots$

We consider the nested sequence of polygons $P_i$ where  $P_i$ has vertices $\bigcup_{1=k}^i [p_k]$. By construction $P_i$ are symmetric and invariant under $\Z/ n\Z$, we denote by $\norm\cdot_i$ the Finsler metric with unit sphere $P_i$. By Corollary \ref{discrete measure}, we get a sequence of discrete measure $\nu_i$ on $\mb S^1$. The following lemma gives a criterion for convergence of $\nu_i$.

\begin{lemma}[{\cite[Theorem 1.3.9]{martelli2016introduction}} ]\label{measure_convergence}
    Let $\nu_i$ be a sequence of measures  on a topological space $X$ such that $\nu_i(K)$ is bounded on every Borel compact set $K\subset X$. Then $\nu_i$ converges up to a subsequence. 
\end{lemma}

We will apply the lemma to the circle $X=\mb S^1$. To apply Lemma \ref{measure_convergence}, it suffices to check the total measure of the sequence $\nu_i$ is bounded. To prove the boundedness, we first give a geometric description of the total measure, which follows directly from Corollary \ref{discrete measure}.
\begin{lemma}\label{circumference}
    Let $P$ be a symmetric star-shaped convex polygon in $\R^2$ which is invariant under $\Z/ n\Z$. Let $\nu$ be the measure on $\mb S^1$ corresponding to $P$ in Corollary \ref{discrete measure}. Then the total measure $\nu(\mb S^1)$ is the perimeter of its dual polygon divided by $4n$.
\end{lemma}
\begin{prop}\label{convergent sequence}
The measures    $\nu_i$ converge up to a subsequence.
\end{prop}

\begin{proof}
    By Lemma \ref{measure_convergence} and Lemma \ref{circumference}, it is enough to show that the perimeter of the dual polygon $P_i^*$ of $P_i$ is uniformly bounded for all $i$. Recall that $P_{i+1}$ is obtained from $P_i$ by adding an equivalence class of vertices. We will show that the perimeter of $P_i$ is monotone decreasing in $i$, by showing that adding a vertex to a polygon $P$ decreases the perimeter of its dual polygon $P^*$.

    We fix the notation as illustrated by Figure \ref{dual polygon}. Let $P_+$ be the polygon obtained adding a  vertex to $P$. The added vertex has two neighboring vertices in $P_+$, which are also vertices of $P$. Denote the three edges incident with the two common vertices by $e_1,e_2,e_3$, which are labeled in the clockwise order. Denote the new edges of $P_{+}$ by $e_4,e_5$. The edge $e_k$ satisfies $a_kx+b_ky=1$. We further denote by $V_{jk}$ the vertex incident with $e_j,e_k$.  In particular, $V_{12}=V_{14}$  is a common vertex of $P$ and $P_{+}$. Finally we denote by $E_k$ the vertex of the dual polygon corresponding to the edge $e_k$.

    If $V_{12}=(x_0,y_0)$, it holds $a_1x_0+b_1y_0=a_2x_0+b_2y_0=1$ and thus $(a_2-a_1)x_0+(b_2-b_1)y_0=0$. In particular, $OV_{12}$ is orthogonal to the edge $E_1E_2$ of the dual polygon. 
    Since $V_{12}=V_{14}$, similarly we can see that $OV_{12}$ is also orthogonal to $E_1E_4$. Thus, $E_1,E_4,E_2$ are on the same line. Moreover, $E_4$ lies between $E_1$ and $E_2$ because of the cyclic ordering of $e_1,e_4,e_5$.

    Similarly, we know that $E_2, E_5, E_3$ are on the same line and $E_5$ is between $E_2$ and $E_3$. This shows that the dual polygon $P_+^*$ is obtained from $P^*$ by truncating the vertex $E_2$ and creating the two new vertices $E_4, E_5$ in the two adjacent sides. 
    By triangle inequality, the perimeter of $P_{+}^*$ is less than that of $P^*$.
\end{proof}

\begin{figure}[h]
    \centering
    \includegraphics[width=0.7\linewidth]{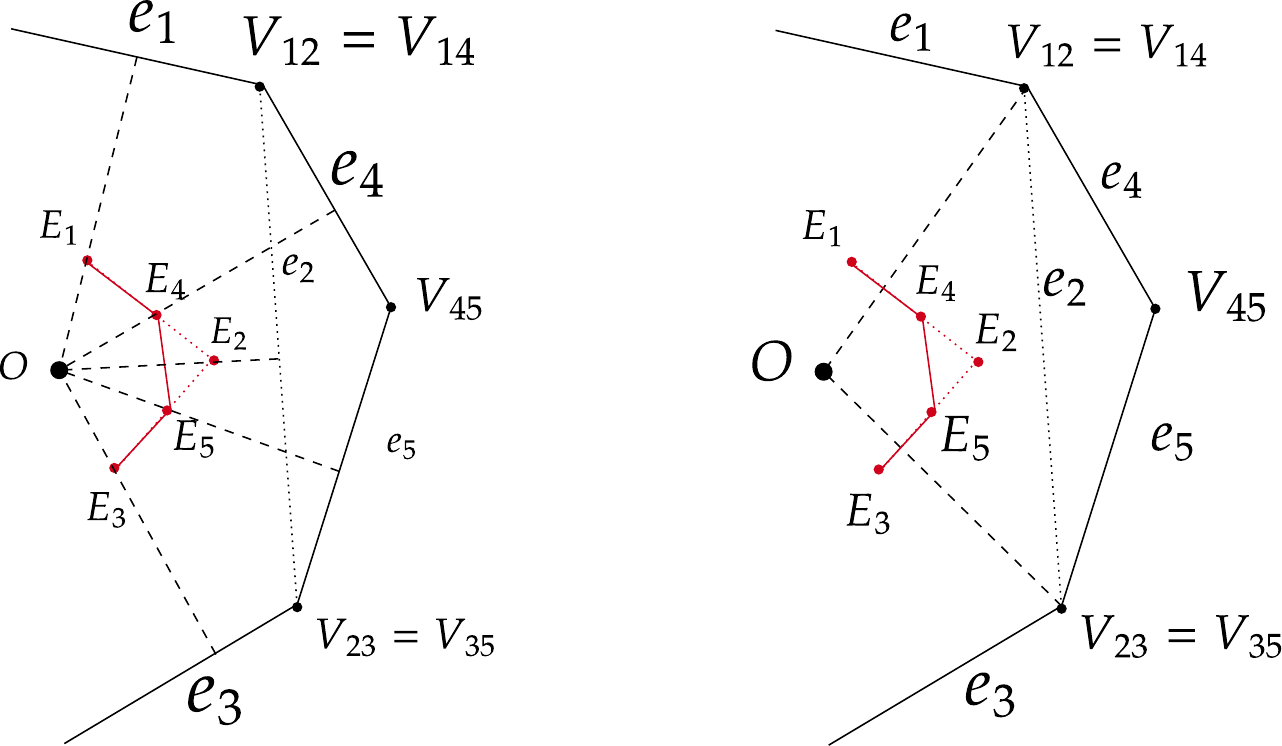}
    \caption{Polygons $P_i,P_{i+1}$ and their dual polygons. The picture on the left shows $OE_i$ must be orthogonal to $e_i$, while the one on the right shows $OV_{jk}$ must be orthogonal to the edge $E_jE_k$.}
    \label{dual polygon}
\end{figure}

\begin{prop}\label{measure_for_vector}
    For every symmetric norm $\norm\cdot$ on $\R^2$ invariant under $\Z/n\Z$, there exists a measure $\nu$ on $\mb S^1$ such that for all $v\in\R^2$ it holds  $\norm v=\int_{\mb S^1}\abs{\innerproduct{v}{v_{[\theta+\frac{\pi}{2}]}}} d\nu$.
\end{prop}

\begin{proof}
    We get from  Corollary \ref{discrete measure}  a sequence of measures $\nu_i$. Proposition \ref{convergent sequence} ensures the existence of a convergent subsequence. We will just assume $\nu_i$ is a convergent sequence. Let $\norm\cdot_i$ be the norm determined by $\nu_i$ and let $\norm v_\nu=\int \abs{\innerproduct{v}{v_{[\theta+\frac{\pi}{2}]}}} d\nu$.
    
    By definition, the measures $\nu_i$ converge  if and only if for every compactly supported function $g$, the integrals $\int g d\nu_i$ converge. Note that for a fixed vector $v$, the function $\theta\mapsto\abs{\innerproduct{v}{v_{[\theta+\frac{\pi}{2}]}}}$ is  compactly supported on $\mb S^1$. Thus  $\norm v_{i}=\int \abs{\innerproduct{v}{v_{[\theta+\frac{\pi}{2}]}}} d\nu_{i}$ converge for every $v$. In other words, the norms $\norm\cdot_i$ converge pointwise to $\norm\cdot_\nu$.

    Let $v,w$ be two vectors. Then\[
    \abs{\innerproduct{v}{v_{[\theta+\frac{\pi}{2}]}}}-\abs{\innerproduct{w}{v_{[\theta+\frac{\pi}{2}]}}}\le \abs{\innerproduct{v-w}{v_{[\theta+\frac{\pi}{2}]}}}\le n\norm{v-w}_E
    \]
    here $\norm{v-w}_E$ is the Euclidean norm. Therefore, $\abs{\norm v_i-\norm w_i}\le n\nu_i(\mb S^1)\norm{v-w}_E$, which implies that $\norm\cdot_i$ is a Lipschitz function with Lipschitz constant $n\nu_i(\mb S^1)$. By Arzela-Ascoli Theorem, $\norm\cdot_i$ converges uniformly to $\norm\cdot_\nu$. In particular, $\norm\cdot_\nu$ is continuous.

    By construction, $\norm\cdot_\nu$ agrees with $\norm\cdot$ on the union of vertices of all polygons, which is a dense subset of the unit sphere. Therefore, $\norm\cdot_\nu=\norm\cdot$.
\end{proof}


\subsection{The Liouville current of a compatible Finsler metric}
We now have all the results needed to prove the central result of the section.
\begin{thm}\label{current_quadratic_differential}
    Given a $1/n$-translation surface with a compatible Finsler metric $F$, there exists a measure $\nu$ on $\mb S^1$ such that \[
    L=\int_{\mb S^1}\mu^s_\theta d\nu
    \]
    is the Liouville current for this Finsler metric. Moreover, the current $L$ is unique.
\end{thm}

\begin{proof}
    The Finsler metric $F$ is induced by a norm $\norm\cdot_F$ on $\R^2$ that is symmetric and invariant under $\Z/n\Z$. Let $\nu$ be the measure associated to $\norm\cdot_F$ in Proposition \ref{measure_for_vector}. We consider the geodesic current $ L=\int_{\mb S^1}\mu_\theta^s d\nu$.

   In order to show that $L$ is indeed a Liouville current for $F$, we need to show that 
   for any closed curve $\gamma$, \[
    i(\delta_\gamma, L)=\int_{\mb S^1}i(\delta_\gamma, \mu^s_\theta)\ d\nu= \ell_F([\gamma]).
    \]
    By Proposition \ref{foliation intersection}, we have $i(\delta_\gamma, \mc F_\theta)=\ell_\theta([\gamma])$. Thus, if we parametrize $\gamma=\gamma(t)$ with $t\in[0,1]$, then 
    \begin{align*}
        \int_{\mb S^1}i(\delta_\gamma, \mc F_\theta)\ d\nu=&\int_{\mb S^1}\ell_\theta([\gamma])\,d\nu \\
        =& \int_{\mb S^1}d\nu\int_0^1\abs{\innerproduct{\gamma'(t)}{v_{[\theta]}}}\,dt\\
        =&\int_0^1dt\int_{\mb S^1}\abs{\innerproduct{\gamma'(t)}{v_{[\theta]}}}\,d\nu \\
        =& \int_0^1 \norm{\gamma'(t)}_F\, dt\\
        =&\ell_F([\gamma]).
    \end{align*}

    By \cite{otal1990spectre}, a geodesic current is uniquely determined by its intersection number with all closed curves, which proves the uniqueness of $L$.
\end{proof}

\bibliographystyle{alpha}
\bibliography{ref}
\end{document}